\documentclass[a4paper,11pt]{amsart}%
\usepackage{hyperref}
\usepackage{pdfsync}
\usepackage{dsfont}
\usepackage{color}
\usepackage{esint}
\usepackage{amsthm}
\usepackage{enumerate}
\usepackage{amsfonts}
\usepackage{amssymb}%

\usepackage{algorithm} 
\usepackage{algpseudocode} 

\usepackage{mathtools}
\usepackage{braket}
\usepackage{bm}
\usepackage{amsmath}
\usepackage{graphicx}
\usepackage{caption}
\usepackage{fullpage}
\usepackage{xcolor}
\usepackage{ulem}
\usepackage{mathtools}
\numberwithin{equation}{section}
\usepackage{float}
\usepackage{subfig}

\usepackage{amsfonts}
\usepackage{graphicx}%
\providecommand{\U}[1]{\protect\rule{.1in}{.1in}}
\newtheorem{theorem}{Theorem}[section]

\newtheorem{corollary}[theorem]{Corollary}

\newtheorem{example}[theorem]{Example}

\newtheorem{lemma}[theorem]{Lemma}

\newtheorem{proposition}[theorem]{Proposition}
\newtheorem{remark}[theorem]{Remark}

\newtheorem{assumption}[theorem]{Assumption}
\definecolor{darkred}{rgb}{0.6,0.1,0.1}
\definecolor{darkgreen}{rgb}{0.1,0.6,0.1}
\definecolor{darkblue}{rgb}{0.1,0.1,0.6}

\DeclareMathOperator{\esssup}{ess\,sup}

\DeclareMathOperator*{\argmax}{arg\,max}

\newcommand{\divergence}{\mathrm{div}}

\newcommand{\F}{\mathcal{F}}

\newcommand{\p}{\varrho}

\newcommand{\R}{\mathbb{R}}

\newcommand{\red}{\nc}

\definecolor{mygreen}{rgb}{0.1,0.75,0.2}
\newcommand{\grn}{\color{green!50!black}}
\newcommand{\nc}{\normalcolor}
\newcommand{\E}{\mathbb{E}}

\newcommand{\sign}{\text{sign}}

\renewcommand{\P}{\mathbb{P}}

\newcommand{\dd}{d}

\title{On the regularized risk of distributionally robust learning over deep neural networks}
\author{Camilo Andr\'es Garc\'ia Trillos}
\address{Department of Mathematics, University College London}
\email{camilo.garcia@ucl.ac.uk}
\author{Nicol\'as Garc\'ia Trillos }
\address{Department of Statistics, University of Wisconsin-Madison}
\email{garciatrillo@wisc.edu}

\begin{document}
\normalem

\begin{abstract}

In this paper we explore the relation between distributionally robust learning and different forms of regularization to enforce robustness of deep neural networks. In particular, starting from a concrete min-max distributionally robust problem, and using tools from optimal transport theory, we derive first order and second order approximations to the distributionally robust problem in terms of appropriate regularized risk minimization problems. In the context of deep ResNet models, we identify the structure of the resulting regularization problems as mean-field optimal control problems where the number and dimension of state variables is within a dimension-free factor of the dimension of the original unrobust problem. Using the Pontryagin maximum principles associated to these problems we motivate a family of scalable algorithms for the training of robust neural networks. Our analysis recovers some results and algorithms known in the literature (in settings explained throughout the paper) and provides many other theoretical and algorithmic insights that to our knowledge are novel. In our analysis we employ tools that we deem useful for a future analysis of more general adversarial learning problems.

\end{abstract}

\maketitle

\section{Introduction}

What is the connection between adversarial learning and regularized risk minimization? This is a question of theoretical and practical relevance that aims at linking two different approaches to enforce robustness in learning models. By an adversarial learning problem here we mean a \textit{distributionally robust optimization} (DRO) problem of the form:
\begin{equation}
  \inf\limits_{\theta \in \Theta } \sup_{\tilde \mu: G(\mu_0,\tilde \mu)\leq \delta } J(\tilde \mu,\theta),
  \label{Robust problem:Intro}
\end{equation}
where $\theta$ denotes the parameters of a statistical learning procedure (for example a neural network, a binary classifier, or the parameters of a linear regression), $\mu_0$ denotes an observed data distribution on $\R^d$, $G$ represents some notion of ``distance" between data distributions, $J(\tilde \mu,\theta)$ is a risk relative to some data distribution $\tilde \mu$ and a loss function $j(x, \theta)$, and finally $\delta$ is a parameter that describes the ``power" of an adversary. On the other hand, by \textit{regularization} we mean an optimization problem of the form
\begin{equation}
 \inf_{\theta \in \Theta} J(\mu_0,\theta) + \lambda R(\theta),
 \label{eqn:regularization}
\end{equation}
where $\lambda>0$ is a positive parameter and $R$ is a regularization functional.

The association between these two types problems has been particularly satisfactory in classical statistical parametric learning settings; see \cite{blanchet2019robust,Blanchet2,OPT-026} and references within. Roughly speaking, if $\theta$ in \eqref{Robust problem:Intro} is a finite dimensional vector representing the parameters of a generalized linear model, \red $J(\mu,\theta) = \int j(x,\theta) d \mu$, where $j$ \nc is a loss function such as square loss or logistic loss, and the function $G$ is an appropriate Wasserstein distance, then one can show that the family of problems \eqref{Robust problem:Intro} coincides with a family of Lasso objectives that includes the popular squared-root Lasso model from \cite{SqRtLasso}.

In more general learning settings, and in particular in the setting explored in this paper, there is no direct equivalence between the adversarial problem \eqref{Robust problem:Intro} and regularization. There are however other settings where one can \textit{exploit} the structure of the given learning model to derive \textit{specific} insights into the type of regularization associated to \eqref{Robust problem:Intro}. For example, in \cite{MurrayNGT}, in the context of non-parametric binary classification, a connection between adversarial learning and regularization is explored by establishing geometric evolution equations that must be satisfied by the ensemble of solutions to the family of adversarial problems \eqref{Robust problem:Intro} indexed by $\delta$. In general, an illuminating strategy that can be followed in a variety of settings in order to gain insights into the regularization counterpart of \eqref{Robust problem:Intro} is to analyze the max part of the problem for small $\delta$ and identify its leading order terms to use them as regularization terms. This is a strategy that has been followed in many works that study the robust training of neural networks, e.g. \cite{ObermanFinlay,GradientRegularization,CurvatureRegularization,SlavinRoss,StructuredGradReg,pmlr-v139-yeats21a}; see more discussion in subsequent sections. However, even with an approximation in hand, the specific structure of the resulting regularization problems will depend on the specific learning models under consideration.

% Except for very specific classical learning settings (see \cite{blanchet2019robust}), the association between adversarial learning and regularization is not explicit, and in general, has remained predominantly motivated by empirical observation and good intuition. 

Our goal in this paper is to provide a concrete mathematical connection between a family of distributionally robust learning problems and regularization problems in the context of deep neural network models, and specifically ResNet models. Our analysis provides new theoretical insights on new and existing methods for robust training of neural networks, suggests new algorithms, and revisits older algorithms that can be recast as specific instances of a general unifying family. Our work also suggests new forms of regularization for optimal control problems that are meaningful beyond the applications to machine learning. 
% In this paper we draw a more precise connection between a problem like \eqref{Robust problem:Intro} with a problem like \eqref{eqn:regularization} for a particular family of choices of $G$ in the regime of small enough $\delta$ and in the setting of large ResNet neural network models. In particular, we identify a regularization function $R$ in \eqref{eqn:Regularization} that can be interpreted as ``first order regularization effect" of the adversarial learning problem \eqref{Robust problem:Intro} (see Theorem \ref{thm:Regularization} for a precise statement). 
The main motivation for working with ResNet models is that there is a clear interpretation of truly deep ResNet models (formally, the number of layers is infinity): in the large number of layers limit, the training of a ResNet may be interpreted as a continuous time optimal control problem. This ODE perspective for understanding and training neural networks has received increased attention in the past few years  -- see \cite{NeuralODESs,Ruthoto,DeepResLimits}. The specific structure of this setting will then allow us to recognize the resulting regularization problems as mean-field control problems, and thus will motivate us to derive their corresponding Pontryagin maximum principles. In turn, these maximum principles can be used to motivate a large class of algorithms for the training of robust networks which includes the double backpropagation algorithm from \cite{DoubleBackProp}. The use of Pontryagin maximum principle based training algorithms has been advocated for in works like \cite{EMaximPrinciple} given their generality, versatility, and theoretical properties.

In the next two sections we introduce the specific set-up that we will work with throughout the paper. Our main theoretical results are presented in section \ref{sec:Regularization} and our algorithms in section \ref{sec:algorithm}.

% In contrast to the YOPO algorithm, our method avoids having to run any data perturbation updates, effectively reducing the running time for the training of robust models. This is because our regularization functional $R$ implicitly captures the worst perturbation in the adversarial problem \eqref{Robust problem:Intro}, at least for small $\delta$.

% derived regularization problem and explore some algorithmic implications for the training of robust networks. 

% \item Through the connection with continuous time control problems, we will be able to explore the regularization effect of min-max problems like \eqref{Robust problem:Intro} in a wider variety of settings beyond machine learning applications. Indeed, unlike other forms of regularization considered in inverse problems and control theory, the one obtained through our analysis motivates some interesting mathematical questions that are partially explored in our sibling paper \cite{}.
% \end{enumerate}

\nc

\medskip

\nc

\subsection{Network models}
\label{sec:NetworkModels}

As discussed in the introduction, we will focus our discussion on deep ResNet neural networks and use the differential equation in $\R^d$
\begin{equation}
   \begin{cases} dX_t=  f(X_t,\theta_t)dt, \quad t \in (0,T),\\ X_0=x \end{cases}  
   \label{eqn:ODE}
\end{equation}
 to model the transformations that an input data point $x \in \R^d$ undergoes along a deep neural network; notice that, with the previous interpretation, $X_T$ is the output of the network when the input is $x$. The function $\theta: [0,T] \rightarrow \Theta_0$ represents the parameters of the network and $\Theta$ is a family of $\theta$s.  The ``time" variable $t$ can be interpreted as index for the layers of the model and the time horizon $T$ as the depth of the network. $\Theta_0$ represents the possible values of parameters at a given layer. We remark once again that a ResNet model found in practice can be seen as a time discretization of \eqref{eqn:ODE}; see \cite{e_meanfield_2018,DeepResLimits}.

\begin{example}
\label{ex:1}
The previous general setting can be used for regression with the following interpretation. We write $x=(v,y) \in \R^{d-1} \times \R$ and interpret $v$ as feature vector and $y$ as label or output. The function $f(\xi, \vartheta )$ can be taken to be
\[f(\xi, \vartheta )=\left( \begin{matrix}\sigma( \vartheta \cdot \xi_{1:d-1}) \\ 0 \end{matrix} \right),  \]
interpreting $\vartheta$ as a $(d-1)\times (d-1)$ matrix, and $\sigma$ as a non-linear function (e.g. a sigmoid or ReLu) that acts coordinate-wisely on $d-1$-dimensional vectors. Notice that if we write $X_t=(V_t,Y_t)^\top$, then $Y_t=y$ for every $t$ and in particular $X_T=(V_T,y)^\top$. 
\end{example}
\nc

We introduce two functions $\ell: \R^d \times \Theta_0 \rightarrow \R$ and $\Phi: \R^d \times \Theta_0$ which from the control theory perspective can be interpreted as terminal and running costs respectively. For $x \in \R^d$ and $\theta \in \Theta$ we define:
\[ j(x,\theta ):= \ell ( X_{x,T}, \theta_T) + \int_0^T \Phi(X_{x,t},\theta_t) dt.  \]
In the above, we have used the notation $X_{x,t}$ to represent the solution to \eqref{eqn:ODE} with the extra subscript highlighting the initial condition $x$. The value of $j(x,\theta)$ can be interpreted as the loss (including the extra penalization) that the network with parameters $\theta$ incurs into when $x$ is the network's input.

\begin{example}
\label{ex:2}
[Continuation of Example \ref{ex:1}]
In the setting considered in Example \ref{ex:1} we can take $\Phi$ to be either identically equal to zero or be a function penalizing the size of the parameters only. As for the terminal cost, we can take
\[\ell(\xi,\vartheta)=  (\vartheta \cdot \xi_{1:d-1}- \xi_d )^2, \]
this time interpreting $\vartheta$ as a $d-1$-dimensional vector, i.e. $\Theta_0$ is a subset of $\R^{d-1}$. With this choice we obtain $\ell(X_{x,T},\theta_T) =( \theta_T \cdot V_{x,T} -y  )^2 $, i.e. squared loss. 
\end{example}
\nc

For a given probability distribution $\mu_0$ and for a control $\theta \in \Theta$ we define the risk
\[ J(\mu_0,\theta):= \E_{x \sim \mu_0} \left[ j(x, \theta) \right],  \]
and consider a so-called mean-field control problem (see \cite{e_meanfield_2018}):
\begin{equation}
\inf\limits_{\theta \in \Theta} J(\mu_0,\theta).
\label{def:control problem}
\end{equation}
% where $(X,Y) \sim \mu_0$ a probability distribution of elements in $\R^d \times \R$, $\Theta \subset L^1([0,T], \R^m) $ %$\Theta \subset L^1([0,T], C(\R^d, \R^d)) $ 
% is non-empty and closed,  $\ell$ is the loss function, $\Phi$ is a running cost that can implement regularization and $h_\theta(x,t)$ solves the ODE
% \begin{equation} h_\theta(x,t)= x + \int_0^t f(h_\theta(x,s), \theta_s ) \dd s, \quad t\in [0,T].
% \label{def:constraints}
% \end{equation}
For us, $\mu_0$ represents the training data distribution which at this stage can be simply assumed to be an empirical measure; problem \eqref{def:control problem} is then a risk minimization problem relative to the training distribution $\mu_0$. We have made the dependence of problem \eqref{def:control problem} on the training data set $\mu_0$ explicit as our goal is precisely to study its sensitivity to changes in the training input as will become more apparent in the next section when we introduce our adversarial learning problem precisely. We remark that in order to rigorously connect the optimization problem characterizing the training of a ResNet model with finitely many layers with the idealized continuous time control problem  model considered here one needs to use variational techniques as discussed in \cite{DeepResLimits}. 

\subsubsection{Further notation} 

\begin{itemize}
\item $x, \tilde x$ represent vectors in $\R^d$ and will be used to denote inputs of the neural network.
\item We use $D$ to denote a matrix of derivatives of a vector valued function whereas we use $\nabla$ to denote the gradient of a scalar valued function. $D^2$ is reserved to indicate matrices/tensors of second order derivatives of scalar/vectorial functions. By convention, we identify the derivative of $A:\R^d\rightarrow \R^{c_1 \times \cdot \times c_k}; x\mapsto A(x)$  with  a tensor of size $(c_1,\ldots,c_k,d)$.
\item $\xi$ represents a vector in $\R^d$ and $\vartheta$ an element in $\Theta_0$. They will be used as dummy variables for the functions $\ell$, $\Phi$, and $f$. In particular, $\nabla_\xi \Phi$ and $\nabla _\vartheta \Phi $ represent the vector of derivatives of $\Phi$ with respect to $\xi$ and $\vartheta$ respectively.

\red 
\item The matrix $D_\xi f$ has coordinates $ [D_\xi f]_{ij} = \frac{\partial }{\partial \xi_j} f_i$, where $(f_1, \dots, f_d)$ are the coordinate functions of $f$. We write the tensor $D^2_\xi f$ in coordinates as
\[ [D^2_{\xi} f]_{ijk} = \frac{\partial^2 f_i}{\partial \xi_j \partial \xi_k}.   \]
Notice that $[D^2_{\xi} f]_{ijk} = [D^2_{\xi} f]_{ikj}$. The tensor $(D^2_{\xi} f)^\top$ is defined as $[(D^2_{\xi} f)^\top]_{ijk} := [(D^2_{\xi} f)]_{kij}$. Finally, the tensor $D^3_\xi f$ is defined in coordinates as
\[ [D^3_{\xi} f]_{ijkl} = \frac{\partial^3 f_i}{\partial \xi_j \partial \xi_k \partial \xi_l},  \]
and we define $[(D^3_{\xi} f)^\top]_{ijkl} := [(D^3_{\xi} f)]_{lijk}$.

\nc
% \item$y$ represents ``labels"; either real numbers (for regression) or elements in a finite set (for classification).

% \item$z$ is the state variable in our generic control problem. We will take it to be either $z=(x,y)$ or $z=(x,\tilde x , y)$.

\item $\theta$ represents the weights of the ``ideal" (continuous time) ResNet neural network. 

% \item $\beta \in\R^d$ and represents the vector used in the last layer of the neural network for either classification or regression.  

 \item $\tilde \mu$ will generically represent a probability measure over the variable $x$ and will typically be interpreted as a perturbation of the data distribution $\mu_0$.
\end{itemize}

\subsection{Adversarial learning}

In this paper we restrict our attention to the family of distributionally robust adversarial problems:
\begin{equation}
  \inf\limits_{\theta \in \Theta } \sup_{\tilde \mu: W_p(\mu_0,\tilde \mu)\leq \delta } J(\tilde \mu ,\theta),
  \label{Robust problem}
\end{equation}
where $W_p$ stands for the $p$-Wasserstein distance:
\[W_p(\mu,\tilde \mu):= \inf_{\pi \in \Gamma(\mu,\tilde \mu)} \left\{ \int_{\R^{d}\times \R^{d}} c_p(x,\tilde x) \dd \pi(x,\tilde x)\right\}^{1/p}\]
defined for two probability measures $\mu, \tilde \mu$ over $\R^d$ (or over a compact subset of $\R^d$). 
The function $c_p: \R^d \times \R^d \rightarrow \R$ is the $p$-cost:
 \[ c_p(x,\tilde x):= \lVert x-\tilde x \rVert^p,\]
 where $\lVert \cdot\rVert$ denotes an arbitrary norm in $\R^d$. The family of problems \eqref{Robust problem} is indexed by $p\in[0,\infty]$. We interpret the case $p=\infty$ as
\[ W_\infty(\mu, \tilde \mu )=  \min_{ \pi \in \Gamma(\mu, \tilde \mu)} \esssup_\pi \{ \lVert x-\tilde x \rVert  \: : \: (x, \tilde x) \in \R^d \times \R^d \}, \]
while the case $p=0$ as the total variation distance between measures. In the remainder we will restrict our attention to the case $p \in [2,\infty]$.

\begin{example}
\label{ex:3}
[Continuation of Example \ref{ex:2}]
While problem \ref{Robust problem} with the cost $c_p$ introduced earlier is meaningful in a regression setting, there are other adversarial settings of interest where for example adversaries are only allowed to perturb feature vectors and not labels. In that case, the cost $c_r$ introduced earlier can be replaced with the closely related cost:
\[ \hat{c}_p( (v,y), (\tilde v , \tilde y) ) = \begin{cases} \lVert v-\tilde v\rVert^p, & \text{if } y=\tilde y, \\ \infty, & \text{ else }.  \end{cases}   \]
The analysis that we present in the remainder of the paper adjusts easily to this cost function. We omit the details.
\end{example}
\nc

\begin{remark}
  We notice that since the Wasserstein distances satisfy the relation $W_p \leq W_{p'}$ when $p \leq p'$, it is straightforward to see that the adversary in problem \eqref{Robust problem} is stronger than the analogous adversary when choosing $p'$. In particular, of all the adversaries indexed by $p$, the weakest one is the one corresponding to $p=\infty$.
\end{remark}

\red 
\begin{remark}
\label{rem:CouplingsRem}
Problem \eqref{Robust problem} can be equivalently reformulated as
\begin{equation}
  \inf\limits_{\theta \in \Theta } \sup_{\pi \in \mathcal{F}_{\mu_0, \delta}} J(\pi ,\theta),
  \label{Robust problem Couplings}
\end{equation}
where $\mathcal{F}_{\mu_0, \delta}$ is the set of probability measures on $\R^d\times \R^d $ satisfying $\int_{\R^d \times \R^d} c_p(x, \tilde x) d \pi(x,\tilde x) \leq \delta^p$ and $P_{1\sharp} \pi = \mu_0$ (where $P_{1 \sharp} \pi$ denotes the marginal of $\pi$ on the first coordinate); we abuse notation slightly and use $J(\pi, \theta)$ to denote $\int_{\R^d \times \R^d} j(\tilde x , \theta)d\pi(x, \tilde x)$. In other words, by replacing the variable $\tilde \mu$ with the variable $\pi$, the non-linear constraint in $\tilde \mu$ gets replaced by linear constraints in $\pi$.
\end{remark}

\begin{remark}
 In an alternative formulation of \eqref{Robust problem}, one can replace the constraint on $\tilde \mu$ with an explicit penalization:
 \begin{equation}
  \inf\limits_{\theta \in \Theta } \sup_{\tilde \mu}\left\{ J(\tilde \mu ,\theta) - \frac{1}{\lambda} W_p^p(\mu_0,\tilde \mu) \right\},
  \label{eqn: Roustness Explicit}
\end{equation}
for some $\lambda>0$. Just as for the constrained problem, \eqref{eqn: Roustness Explicit} admits an equivalent reformulation in terms of couplings: 
 \begin{equation*}
  \inf\limits_{\theta \in \Theta } \sup_{P_{1\sharp} \pi = \mu_0}\left\{ J(\pi ,\theta) -\frac{1}{\lambda} \int_{\R^d \times \R^d} c_p(x, \tilde x) d \pi(x, \tilde x)  \right\}.
\end{equation*}
Since $J$ is linear in $\pi$, under additional convexity assumptions on the set $\Theta$ and on the function $j(x, \cdot)$, up to technical details the above problem is equivalent to \eqref{Robust problem Couplings} in the sense that for a given $\delta>0$ there is a $\lambda>0$ such that a solution to one problem is also a solution for the other. Indeed, under these additional convexity assumptions, the $\inf$ and $\sup$ can get swapped in \eqref{Robust problem Couplings}, which in turn would allow us to replace the explicit constraint on $\pi$ with an explicit penalization for a fixed constant $\lambda>0$.

Without the additional convexity assumptions mentioned above the two problems may very well be different. It is however important to notice that if we attempted to find an approximation to problem \eqref{eqn: Roustness Explicit} for small $\lambda$ following some of the steps in section \ref{sec:TaylorGeometry} we would obtain a problem that is very similar to \eqref{eqn:Regularization}, which recall is an approximation to the adversarial problem with explicit data perturbation constraints.
\end{remark}

\nc

\subsection{Regularized risk minimization and associated control problems}
\label{sec:Regularization}

In order to make our results mathematically precise, we impose some extra conditions on all the terms that determine the min-max problem \eqref{Robust problem} in our setting.

\begin{assumption}\ 
  \begin{enumerate}[i.]
      \item $f$ is bounded; $f, \Phi$ are continuous in $\theta$; and $f, \Phi, \ell$ are continuously differentiable with respect to $x$. The first derivatives of $f, \Phi, \ell$ are Lipschitz continuous in space uniformly in $\theta$.     
      Moreover, $\ell$ is continuously differentiable with respect to $\theta$.
      \item The distribution $\mu_0$ has bounded support in $\R^{d}$.
  \end{enumerate}
  \label{A1}
  \end{assumption}

\begin{assumption}\ 
  \begin{enumerate}[i.]
      \item Assumption \ref{A1} holds; and
      \item $f, \Phi, \ell$ are twice continuously differentiable with respect to $x$. 
  \end{enumerate}
  \label{A2}
  \end{assumption}

% \red 
% \begin{remark}
% In the context of Remark \ref{}, if the set $\Theta_0$ is bounded, the functions are .... and for example the terminal cost function is quadratic, then Assumption \ref{assump:Smoothness} is satisfied.
% \end{remark}
% \nc

\nc 
Our first theoretical result is the following.

\begin{theorem}(Regularization of distributionally robust adversarial learning: first order case)
\label{thm:Regularization}
Let $p \in [2,\infty]$. Under Assumptions \ref{A2} the objective function
\begin{equation}
\sup_{\tilde \mu: W_p(\mu_0,\tilde \mu)\leq \delta } J(\tilde \mu, \theta) 
\label{eqn:RobustObjective}
\end{equation}
is equal to
 \[ J(\mu_0, \theta) +  \delta \cdot \left(\E_{x \sim \mu_0} \left[  \lVert\nabla_x j(x,\theta)\rVert^q_* \right]\right)^{1/q} + O(\delta^2), \]
where the $O(\delta^2)$ term is uniform over all $\theta \in \Theta$, $q$ is $p$'s conjugate, i.e. $\frac{1}{p}+\frac{1}{q}=1$, and $\lVert\cdot \rVert_*$ is the dual norm of $\lVert\cdot \rVert$. In particular,
\[ V_\delta^* - U_\delta^* = O(\delta^2), \]
where $V_\delta^*$ is equal to the infimum of \eqref{eqn:RobustObjective} over all $\theta \in \Theta$ and $U_\delta^*$ is equal to
\begin{equation}
\label{eqn:Regularization}
\inf_{\theta \in \Theta} \left\{ J(\mu_0, \theta) + \delta \cdot  \left(\E_{x \sim \mu_0} \left[  \lVert\nabla_x j(x,\theta)\rVert^q_* \right]\right)^{1/q} \right\}.
 \end{equation}
\end{theorem}

\red 

\begin{remark}
  Theorem \ref{thm:Regularization} implies that under the given assumptions, for any fixed control in the admissible domain, the value function of the robust problem and the value function of the regularized problem are close with an error of order $\delta^2$. Hence, a minimizer for any of the two problems is a $\delta^2$-minimizer for the other.
\end{remark}
\nc

\begin{remark}
\label{rem:p=inftyFirst}
When $p=\infty$, problem \eqref{eqn:Regularization} can be written more succinctly as
\begin{equation}
\min_{\theta \in \Theta } \left\{ J(\mu_0, \theta) + \delta \cdot \E_{x \sim \mu_0} \left[ \lVert \nabla_x j(x,\theta)\rVert_*  \right]\right\}.
\label{Robust control infty}
\end{equation}  
Notice that the objective is linear in $\mu_0$. 
This property is useful in connection to the use of stochastic gradient methods for training; see the discussion in section \ref{sec:algorithm}. Problem \eqref{Robust control infty} and other closely related problems that are linear in $\mu_0$ and that penalize the gradient of the loss function have been considered in several works in the literature including \cite{DoubleBackProp,ObermanFinlay,GradientRegularization,SlavinRoss,StructuredGradReg,pmlr-v139-yeats21a}.
\end{remark}

It is important to highlight that Theorem \ref{thm:Regularization} does not depend on the specific structure of neural network models and indeed continues to be true in other learning settings as long as the following assumption (implied by Assumption \ref{A2} in our neural network setting) holds:

\begin{assumption}
\label{assump:Smoothness}  
For every $\theta \in \Theta$ the function $j(\cdot,\theta)$ is twice differentiable with uniformly bounded second derivatives:
\[ \lVert  D^2_x j(\cdot ,\theta) \rVert_{L^\infty(\R^d)} \leq C, \]
where $C>0$ is independent of $\theta \in \Theta$.
\end{assumption}

The interpretation of Theorem \ref{thm:Regularization} is straightforward: the leading order regularization effect of the adversarial problem \eqref{Robust problem} with $p$-cost is given by \eqref{eqn:Regularization}. Put in another way, \eqref{Robust problem} for small $\delta$ is an approximation to problem \eqref{eqn:Regularization}. Moreover, a minimizer of \eqref{eqn:Regularization} is guaranteed to have value of \eqref{eqn:RobustObjective} within $O(\delta^2)$ of the minimum. Notice that Problem \eqref{eqn:Regularization} is a form of regularized risk minimization of the form \eqref{eqn:regularization} with regularization term given by
\[ \left( \E_{x \sim \mu_0} \left[  \lVert\nabla_x j(x,\theta)\rVert_*^q \right]\right)^{1/q}, \]
only that in this case the regularization does depend on the data distribution $\mu_0$. Intuitively, this new term forces the parameters of the model to be chosen so as to make the loss $j( x,\theta)$ insensitive to perturbations of the data (i.e. the gradient in $x$ of the loss function should be small in the support of $\mu_0$). 

\medskip
In our setting of interest, the regularized problem \eqref{eqn:Regularization} possesses an interesting structure that will be convenient to elaborate on as it provides the basis for the algorithms for training robust neural networks discussed in section \ref{sec:algorithm}. 
% The adversarial learning problem \eqref{Robust problem} is thus tied with a family of algorithms for the training of neural networks that have been introduced in the literature and that rely on Pontryagin's maximum principle as a more general framework than backprogation (see \cite{EMaximPrinciple}).
First, let us consider the control problem:
\begin{equation}
\begin{aligned}
\inf_{\theta \in \Theta } \quad & \E_{x \sim \mu_0}\left[ j(x,\theta) \right] \\
\textrm{s.t.} \quad & \begin{cases}dX_{x,t} =f(X_{x,t},\theta_t) dt , \quad t \in (0,T), \\ X_{x,0}=x,    \end{cases}
\end{aligned}
\label{eqn:ControlBasic}
\end{equation}
which is nothing but the (unrobust) training problem for the neural network; we have made the constraints in the above problem explicit to facilitate the comparison with the problems introduced below. For each $\theta \in \Theta$ and $x \in \R^d$ there is a corresponding dual variable $P :[0,T] \rightarrow \R^d$ associated to the ODE \eqref{eqn:ODE} which can be written as: 
\begin{equation}
\begin{cases}dP_{x,t} = \red- \nc\nabla_\xi H_0(X_{x,t},\theta_t,P_{x,t}) dt , \quad t \in (0,T), \\ P_{x,T}=- \nabla_\xi \ell(X_{x,T},\theta_T), 
\end{cases}
\end{equation}
where $H_0$ is the Hamiltonian: 
\[ H_0(\xi,\vartheta, \varrho):= \varrho \cdot f(\xi,\vartheta)  -  \Phi(\xi,\vartheta), \quad  \xi \in \R^d, \quad \vartheta \in \Theta_0, \varrho \in \R^d.\]
$P_{x,0}$ is known to be equal to the \red negative \nc gradient (in $x$) of the function $j(\cdot,\theta) $ when holding $\theta$ fixed (see the beginning of section \ref{Sec: Locally robust problem} for a proof of this fact), that is, $P_{x,0}$ captures the sensitivity of $j(x,\theta)$ to perturbations in the input. This insight has been used in \cite{YOPO} to propose algorithms for the training of robust neural networks.

With this new interpretation, Problem \eqref{eqn:Regularization} can be rewritten as a control problem:
\begin{equation}
\begin{aligned}
\inf_{\theta \in \Theta } \quad & \left\{ \E_{x \sim \mu_0}\left[ j( x,\theta)  \right]+ \delta \cdot\left(  \E_{x \sim \mu_0}  \left[  \lVert P_{x,0} \rVert_*^q \right]\right)^{1/q} \right\} \\
\textrm{s.t.} \quad & \begin{cases}dX_{x,t} =f(X_{x,t},\theta_t) dt , \quad t \in (0,T) \\ X_{x,0}=x\\
dP_{x,t} = \red- \nc \nabla_\xi H_0(X_{x,t}, \theta_t, P_{x,t}) dt , \quad t \in (0,T) \\ P_{x,T}=- \nabla_\xi \ell(X_{x,T},\theta_T) .
\end{cases}
\end{aligned}
\label{eqn:ControlExpanded}
\end{equation}

Our second main result identifies the first order optimality condition for this new control problem.

\begin{theorem}
  Let $p \in [2,\infty]$. Suppose that Assumption \ref{A2} holds and that $\lVert \cdot \rVert$ is the Euclidean norm. Suppose that $(\theta^*, X^*, P^*)$ is a minimizer for problem \eqref{eqn:ControlExpanded}. Then, there exist absolutely continuous processes $\alpha^*,\beta^*$ such that for $\mu_0$-a.e. $x$ we have:
  \begin{align}
    \alpha_{x,t}^* & =  -\nabla_\xi \ell(X_{x,T}^*,\theta_T^*) + D^2_\xi \ell(X_{x,T}^*,\theta_T^*)^\top \beta^*_{x,T} \label{Def alpha}\\
    &\qquad  + \int_t^T \{ D_\xi f(X_{x,s}^*,\theta_s^*)^\top \alpha^*_{x,s} - \nabla_\xi \Phi(X_{x,s}^*,\theta_s)   \} \dd s \nonumber\\
    &\qquad  + \int_t^T \{ D^2_\xi \Phi(X_{x,s}^*,\theta_s^*)  - D^2_\xi f(X_{x,s}^*,\theta_s^*)^\top P_{x,s}^*   \}^\top  \beta_{x,s}^* \dd s \nonumber\\
    \beta_{x,t}^* & =   \delta  \nc \left(  \E_{x_0 \sim \mu_0} \left[ \lVert P^*_{x,0}\rVert^q  \right] \right)^{- \frac 1 p }  \lVert P_{x,0}^*\rVert ^{q-2} P_{x,0}^*  + \int_0^t D_\xi f(X_{x,s}^*,\theta_s^*) \beta_{x,s}^* \dd s, \label{Def beta}
  \end{align}
  and also
  \begin{equation}
    \theta_t^* \in  \arg\max_{\vartheta \in \Theta_0}  \left\{  \E_{x \sim \mu_0} \left[ H(X_{x,t}^*,P^*_{x,t},\alpha_{x,t}^*, \beta_{x,t}^*,\vartheta)\right]    \right\}, 
    \label{Modified maximum principle}
  \end{equation} 
  where 
  \begin{equation}
    H(\xi,\p,\alpha,\beta,\vartheta) := \red \alpha \cdot f(\xi,\vartheta) -\Phi(\xi,\vartheta) -   \beta \cdot ( D_\xi f(\xi,\vartheta)^ \top \p   -\nabla_\xi \Phi(\xi,\vartheta)  ) \nc
    \label{New Hamiltonian}
  \end{equation} 
  is the Hamiltonian of problem \eqref{eqn:ControlExpanded}.
  \label{Thm: Pontryagin}
  \end{theorem}  
  \begin{remark}
    It is worth highlighting that for the modified problem, the adjoint variable of $X^*$ is \emph{not} $P^*$, but rather $\alpha^*$. Indeed, for this problem, $P^*$ becomes part of the state variables, and has its own adjoint variable $\beta^*$. The Hamiltonian name for $H$ is then justified since
    \begin{align*}
      \dd X^*_{x,t} & =  \nabla_\alpha H(X^*_{x,t}, P^*_{x,t}, \alpha^*_{x,t}, \beta^*_{x,t}, \theta^*_t ), \\  
      \dd P^*_{x,t} & =   \nabla_\beta H(X^*_{x,t}, P^*_{x,t}, \alpha^*_{x,t}, \beta^*_{x,t}, \theta^*_t ), \\
      \dd \alpha^*_{x,t} & = \red - \nc \nabla_\xi H(X^*_{x,t}, P^*_{x,t}, \alpha^*_{x,t}, \beta^*_{x,t}, \theta^*_t ), \\      
       \dd \beta^*_{x,t} & = \red - \nc \nabla_p H(X^*_{x,t}, P^*_{x,t}, \alpha^*_{x,t}, \beta^*_{x,t}, \theta^*_t ).
    \end{align*}
  \end{remark}

 \begin{remark}
 \label{rem:GeneralNorms}
  A similar result can be derived for more general norms, only that the expressions for the corresponding adjoint variables are more cumbersome.
 \end{remark}

The derived Pontryagin principle motivates a class of algorithms for training robust neural networks that are discussed in section \ref{sec:algorithm}. The double backpropagation algorithm from \cite{DoubleBackProp} is a particular instance of this family of algorithms, which, at the moment it was proposed, was used to enhance the generalization properties of a neural network.

% has an additional advantage over the YOPO algorithm. We will present the algorithm in section \ref{} but at this point we can provide some insights on the benefits that this algorithm may have over, say, the YOPO algorithm. Notice that the YOPO algorithm aims at solving the original min-max problem having to update both the parameters of the neural network as well as the data. Our algorithm is on the other hand based on the regularization problem \eqref{} which in a sense has already computed the max part of the min problem. This means that our algorithm does not have to perturb the data, and thus is expected to enjoy of better stability properties. 

\nc

\red 

\subsubsection{Second order regularization}
\label{sec:SecondOrderRegResults}
\nc

The results presented in the previous section can be developed further to include higher order expansions for the function $J(\tilde \mu, \theta)$ under the assumption that the function $j(\cdot ,\theta)$ is regular enough. However, unlike in the first order case, an explicit higher order expansion for \eqref{eqn:RobustObjective} that does not involve any maximization problems is in general difficult to obtain unless one restricts to specific regimes for the size of the gradient function $\nabla_x j(\cdot ,\theta)$ relative to the size of the parameter $\delta$. Nevertheless, restricting our attention to the case $p\geq 2$ and $\lVert\cdot \rVert$ the Euclidean norm, in section \ref{sec:SecondOrderReg} we can motivate the following optimization problems:
\begin{equation}
\label{eqn:SecondOrderVersion1}
\min_{\theta \in \Theta} \left\{ \begin{split}  J(\mu_0, \theta) & + \delta \left(\E_{x \sim \mu_0} \left[ \lVert \nabla_x j(x,\theta)\rVert^q \right] \right)^{1/q} \\ & +  \frac{\delta^2}{2} \left( \E_{x \sim \mu_0} \left[ \lVert \nabla_x j(x,\theta)\rVert^q \right] \right)^{-2/p} \left(\E_{x \sim \mu_0} \left[  \frac{\nabla_x  j^\top D^2_x j\nabla_x j}{ \| \nabla_x  j \|^{2(1-1/(p-1))}} (x,\theta) \right] \right) \end{split} \right\},
\end{equation}
\begin{equation}
\label{eqn:SecondOrderVersion2}
\min_{\theta \in \Theta} \left\{  J(\mu_0, \theta) +  \frac{\delta^2}{2} (\E_{x \sim \mu_0} \left[ (\lambda_{max}(x,\theta))_+^{\tilde q} \right])^{1/\tilde{q}} \right\},
\end{equation}
where $\tilde{q}$ is the conjugate of $p/2$, i.e. $\frac{1}{\tilde q} + \frac{2}{p}=1$. Both of these optimization problems can be regarded as second order regularized risk minimization problems stemming from the adversarial learning problem \eqref{Robust problem}, the difference between them being the relative size expected for the gradients of the loss function as compared to $\delta$. As in Remark \ref{rem:GeneralNorms}, similar second order problems can be motivated for more general norms $\lVert \cdot \lVert$, but we skip the details for concreteness.

It is important to highlight that problems \eqref{eqn:SecondOrderVersion1} and \eqref{eqn:SecondOrderVersion2} are closely connected to other problems in the literature that have been used to train robust neural networks, most prominently the \textit{curvature regularization} problem introduced in \cite{CurvatureRegularization}. To draw a closer connection between what we do here and what is done in \cite{CurvatureRegularization}, notice that, as for the first order case, when $p=\infty$ problems \eqref{eqn:SecondOrderVersion1} and \eqref{eqn:SecondOrderVersion2} are linear in $\mu_0$ and read, respectively, 
\begin{equation}
\label{eqn:SecondOrderVersion1p=infty}
\min_{\theta \in \Theta} \left\{  J(\mu_0, \theta) + \delta \E_{x \sim \mu_0} \left[ \lVert \nabla_x j(x,\theta)\rVert \right] +  \frac{\delta^2}{2} \E_{x \sim \mu_0} \left[  \frac{\nabla_x  j^\top D^2_x j\nabla_x j}{ \| \nabla_x  j \|^{2}} (x,\theta) \right]  \right\},
\end{equation}
\begin{equation}
\label{eqn:SecondOrderVersion2p=infty}
\min_{\theta \in \Theta} \left\{  J(\mu_0, \theta) +  \frac{\delta^2}{2} \E_{x \sim \mu_0} \left[ (\lambda_{max}(x,\theta))_+ \right] \right\}.
\end{equation}
Problem \eqref{eqn:SecondOrderVersion2p=infty} is closely related to an optimization problem introduced in \cite{CurvatureRegularization}, which effectively uses a regularization term of the form:
\[ \E_{x \sim \mu_0} \left[ \lVert D^2_xj(x,\theta) \rVert^2  \right]  \]
(where the matrix norm in the above expectation is the operator norm) instead of the quadratic term in \eqref{eqn:SecondOrderVersion2p=infty}. Notice that in \eqref{eqn:SecondOrderVersion2p=infty} curvature is only penalized when it is positive. This is reasonable as directions with positive curvature are precisely those that an adversary can use to increase the value of the loss function. 

Problem \eqref{eqn:SecondOrderVersion1p=infty} can also be motivated by considerations discussed in \cite{CurvatureRegularization}. Indeed, according to \cite{CurvatureRegularization}, in settings of interest involving the use of neural networks (see Remark 3 in \cite{CurvatureRegularization}) the inner product between the eigenvector corresponding to the maximum eigenvalue of the Hessian matrix $D^2_xj(x,\theta)$ and the \textit{sign gradient} direction are found to have a large inner product, suggesting that this direction is (almost) parallel to the direction of largest curvature. Note that the quadratic term in \eqref{eqn:SecondOrderVersion1p=infty} is precisely the effect of the second derivative along unitary vectors in the direction of the gradient and coincides with the (1-homogeneous) \textit{$\infty$-Laplacian} of the function $j(\cdot, \theta)$. The works \cite{NEURIPS2018_851ddf50,NEURIPS2018_803a82de} also suggest that gradient directions are directions where the loss function $j(\cdot, \theta)$ is highly curved. From these observations it is thus reasonable to consider the second order regularization term that appears in \eqref{eqn:SecondOrderVersion1p=infty}. 

% Our analysis thus provides further theoretical basis for the use of regularization problems for the training of robust neural networks and ties these problems with distributionally robust adversarial problems, i.e. min-max problems like \eqref{Robust problem}.

 \medskip

As in the first order case, we can also study the optimal control structure that the objectives in \eqref{eqn:SecondOrderVersion1} and in \eqref{eqn:SecondOrderVersion2} posses.  Similarly to the first order robustness case, we write the control problem in a different form by expressing the second order part in terms of dual variables.

% where 
% \begin{equation}
%   \hat \alpha_{x,t} = D^2_\xi \ell(X_{x,T},\theta_T)^\top \beta_{x,T}  + \int_t^T  D_\xi f(X_{x,s},\theta_s)^\top \hat \alpha_{x,s}  \dd s   - \int_t^T \{ D^2_\xi f(X_{x,s},\theta_s)^\top P_{x,s}   \}^\top  \beta_{x,s} \dd s 
%   \label{alpha hat}
% \end{equation}
%  In the equations above, we drop the $*$ on the variables to highlight that the dual variables can be defined in this way even for non-optimal paths.

Let us consider first the second order problem \eqref{eqn:SecondOrderVersion1}. We can readily rewrite it in terms of the adjoint variables of the first order expansion. 

\begin{proposition}
  Under the same assumptions as Theorem \ref{Thm: Pontryagin},  Problem \eqref{eqn:SecondOrderVersion1} can be rewritten as 
\begin{equation}
  \begin{aligned}
  \inf_{\theta \in \Theta } \quad & \left\{ \E_{x \sim \mu_0}\left[ \ell( X_{x,T},\theta_T)  \right]+ \delta \cdot \left( \E_{x \sim \mu_0} \left[  \lVert P_{x,0}\rVert^q \right]\right)^{1/q} + \frac{\delta}{2} \frac{\E_{x \sim \mu_0} \left[  \lVert P_{x,0}\rVert^{q-2} \hat{\alpha}_{x,0}\cdot P_{x,0}  \right] }{ \left(\E_{x \sim \mu_0} \left[  \lVert P_{x,0}\rVert^q \right]\right)^{1/p} }     \right\} \\
  \textrm{s.t.} \quad & 
      \begin{cases} 
        X,P \text{ as in \eqref{eqn:ControlExpanded}}\\
        \beta \text{ as in Theorem \ref{Thm: Pontryagin}}   \\
      \dd \hat{\alpha}_{x,t}^*  = \red - \nc D_\xi f(X_{x,s},\theta_s)^\top \hat {\alpha}_{x,s}  \red + \nc \{ D^2_\xi f(X_{x,s},\theta_s)^\top P_{x,s}   \}^\top  \beta_{x,s} \\
      \hat {\alpha}_{x,T} = D^2_\xi \ell(X_{x,T},\theta_T)^\top \beta_{x,T}. 
  \end{cases}
  \end{aligned}
  \label{eqn:ControlExpandedSecondOrderWithAdjoints}
  \end{equation}
  \label{Prop:SecondOrder1WithAdjoints}
\end{proposition}

The main advantage of writing this case of second order problem as in \eqref{eqn:ControlExpandedSecondOrderWithAdjoints} is that this writing avoids the need to keep track of the whole matrix of second derivatives $D_x^2 j(x,\theta)$ which is usually expensive to calculate when the dimension of the problem is large, i.e. $O(d^2$) as opposed to the cost for keeping track of the gradient only which is $O(d)$).

\begin{remark}
  Thanks to the linearity of the adjoint variable $\alpha$ in Theorem \ref{Thm: Pontryagin} it is possible to write
    \[ \alpha_{x,t} = P_{x,t} + \hat \alpha_{x,t},\]
  so that we could have introduced directly $\hat \alpha$ as the adjoint variable of interest. 
  \end{remark}
\begin{remark}
  Note that there is no mistake in the powers of $\delta$ in the objective function in \eqref{eqn:ControlExpandedSecondOrderWithAdjoints}: the adjoint variable $\hat \alpha$  is already scaled by $\delta$ via the initial condition in $\beta$.
\end{remark}
Another advantage of the formulation based in control variables is that, as in the first order case, we can use the optimal control tools to deduce a Pontryagin principle: the adjoint variables $\beta, \hat \alpha$ are added to the list of primal variables, and new adjoint variables $\phi,\pi,\lambda,\psi$ are found. In this case the system will be composed by eight variables (four primal and four adjoint), all having the same dimension $d$. Once more, this shows the complexity of the training problem for this second order case has increased by 'only' a factor of 2 (i.e., a factor independent of dimension).

As an example, we state, without proof, Pontryagin's principle for the problem  \eqref{Prop:SecondOrder1WithAdjoints} in the case $p=\infty$.

\begin{theorem}
  Suppose that Assumption \ref{A2} holds, and further
  that  $f, \ell$ are three-times continuously differentiable with respect to $x$. Assume also that $(\theta^*, X^*, P^*, \alpha^*, \beta^*)$ is an optimal minimizer for problem \eqref{eqn:ControlExpandedSecondOrderWithAdjoints} (with $p=\infty$). Then, there exist absolutely continuous processes $\phi^*,\pi^*, \lambda^*, \psi^*$ such that for $\mu_0$-a.e. $x$ we have:
  \begin{align*}
    d \phi_{x,t}^* & =   \red - \nc D_\xi f(X_{x,t}^*,\theta_t^*)^\top \phi_{x,t}^*  \red + \nc \{D^2_\xi f(X_{x,t}^*,\theta_t^*)^\top P_{x,t}^*   \}^\top  \pi_{x,t}^* \\
    & \qquad    \red + \nc \{D^2_\xi f(X_{x,t}^*,\theta_t^*)^\top \alpha_{x,t}^*   \}^\top  \lambda_{x,t}^* -  \{D^2_\xi f(X_{x,t}^*,\theta_s^*)^\top \beta_{x,t}^*   \}^\top  \psi_{x,t}^* \\
    & \qquad - (\{D^3_\xi f(X_{x,t}^*,\theta_t^*)^\top P_{x,t}^*   \}^\top \beta_{x,t}^*)^\top  \lambda_{x,t}^* dt \\
    d \pi_{x,t}^* & =   D_\xi f(X_{x,t}^*,\theta_t^*) \pi_{x,t}^*  - \{ (\beta^*_{x,t})^\top (D_\xi^{\red 2 \nc} f(X_{x,t}^*,\theta_t^*))^\top  \}^\top \lambda^*_{x,t} d t,\\
    d \lambda_{x,t}^* & =  D_\xi f(X_{x,t}^*,\theta_t^*)^\top \lambda_{x,t}^* d t,  \\
    d \psi_{x,t}^* & =   \red - \nc \{D^2_\xi f(X_{x,t}^*,\theta_t^*)^\top P_{x,t}^*   \}^\top  \lambda_{x,t}^* \red - \nc D_\xi f(X_{x,t}^*,\theta_t^*) \psi^*_{x,t} d t 
  \end{align*}
  with boundary conditions
  \begin{align*}
     \phi_{x,T}^* & = P_{x,T}^* \red + D^2_\xi \ell(X^*_{x,T}, \theta^*_{T}) \pi^*_{x,T}  \nc  -\{D_\xi^3 \ell(X^*_{x,T}, \theta^*_{T}) \beta_{x,T}^* \}^\top \lambda^*_{x,T} \\ 
     \pi_{x,0}^* & = \frac 1 {\lVert P^*_{x,0}\rVert} \left( \red P^*_{x,0} \nc +  \frac \delta 2 \alpha^*_{x,0} - \red \delta \nc \psi^*_{x,0}  \right) - \frac {\red \delta \nc } {\lVert P^*_{x,0}\rVert^3} P_{x,0}^* \cdot \left( \frac {\red 1 \nc} 2 \alpha^*_{x,0} - \psi^*_{x,0}  \right)  P_{x,0}^*
      \\ 
     \lambda_{x,0}^* & = \frac \delta 2 \frac{P^*_{x,0}}{\lVert P^*_{x,0}\rVert} \\ 
     \psi_{x,T}^* & = - D_\xi^2\ell(X_{x,T}^*,\theta_T^*)^\top \beta^*_{x,T}
  \end{align*}
  
  and also
  \[
    \theta_t^* \in \arg\max_{\vartheta \in \Theta_0} \left\{  \E_{x \sim \mu_0} \left[ \bar H(X_{x,t}^*,P^*_{x,t},\alpha_{x,t}^*, \beta_{x,t}^*,\phi_{x,t}^*,\pi_{x,t}^*,\lambda_{x,t}^*,\psi_{x,t}^*,\vartheta)\right]    \right\}, 
    \]
  where 
  \begin{equation}
    \begin{split}
      \bar H(\xi,\p,\alpha,\beta, \varphi, \varpi, \lambda,\varPsi,  \vartheta) & :=  \varphi^\top  f(\xi,\vartheta) \red - \nc \p^\top  D_\xi f(\xi,\vartheta) \varpi \red - \nc  \alpha^\top D_\xi f(\xi,\vartheta) \lambda    \\
      & \qquad + \lambda^\top ( D^2_\xi f(\xi,\vartheta)^\top \p )^\top \beta  +  \beta^\top D_\xi f(\xi,\vartheta) \psi 
    \end{split}
  \end{equation} 
  is the Hamiltonian of problem \eqref{eqn:ControlExpandedSecondOrderWithAdjoints}

  \label{Thm: Pontryagin second order p infty}
  \end{theorem}

\red 
\begin{remark}
Algorithm \ref{Algo 2} in section \ref{sec:algorithm} is an optimization algorithm based on the Pontryagin principle presented in Theorem \ref{Thm: Pontryagin second order p infty} for the second order regularization problem \eqref{eqn:SecondOrderVersion1p=infty}. Ignoring the $\delta$ term in the objective, and considering $j(x,\theta)= \ell(X_{x,T},\theta_T)$, the resulting problem is related to the curvature regularization algorithm studied in \cite{CurvatureRegularization}. In what follows we describe a difference and a similarity between the iterative schemes that we present here and the scheme presented in \cite{CurvatureRegularization} for the analogous problem.

First, notice that the term $D^2_xj\nabla_x j $, which appears in the objective in \eqref{eqn:SecondOrderVersion1p=infty}, can be interpreted as the derivative of $\nabla_x j$ in the direction $z=\nabla_x j$. It is clear that $z$ depends on $\theta$ and thus it is reasonable (and more accurate) to track this dependence when optimizing over $\theta$, as we do in our schemes. Notice that our adjoint equations precisely contain the information to carry out this ``chain rule" computation.  In contrast, at each iteration of the algorithm in \cite{CurvatureRegularization} the value of $z$ gets fixed using the control $\theta$ from the previous iteration before proceeding to take a gradient step in the parameters of the network. In fact, in \cite{CurvatureRegularization} a finite difference approximation $\frac{1}{h}\lVert \nabla_x j(x+ hz , \theta) - \nabla_x j(x,\theta)   \rVert$ for the second derivative is considered, as opposed to the explicit computation $D^2 z$. Our scheme based on Pontryagin's principle does not incur in a higher computational cost than the algorithm in \cite{CurvatureRegularization}. 

On the other hand, our schemes based on Pontryagin principles share with that in \cite{CurvatureRegularization} the fact that data points are never updated during training, in contrast to other works like \cite{YOPO,Goodfellow1} where, at each iteration of their algorithms, data get modified and then used as \textit{the} data for one step of gradient descent in the parameters of the network; see the discussion in section \ref{sec:PerturbationRobustAlgorithm}. 
\end{remark}
\nc

We can now consider the second order problem \eqref{eqn:SecondOrderVersion2}. Unfortunately, it is not simple to express the maximum eigenvalue or its associated eigenvector in terms of a variable having reasonable dynamics in terms of the problem's data without tracking the full second order derivative $D^2_x j(\theta,x)$. As discussed earlier, choosing such a path would be costly for high dimensional problems.

A common idea to deal with dimensionality issues is to introduce random sampling: for instance, we can consider the action of the second derivative on a sample of directions taken uniformly from the unit ball. Indeed, using the fact that
\begin{equation}
  (\lambda_{max}(x,\theta))_+ = \max_{z \in S^d_1} \{ (\langle Az, z \rangle )_+    \},    
  \label{lambda_max as l_infinity}
\end{equation}
where $S^d_1$ denotes the surface of the unit ball in dimension $d$, we could be tempted to consider as candidate 
\[ \max_{i = 1, \ldots, m } \{ (\langle AZ_i, Z_i \rangle )_+ \} ; \qquad \{Z_i\}_{i=1,\ldots, m} \text{ i.i.d with } Z_1\sim U(S^d_1).\]
However, one can easily deduce that this technique requires up to $O(\epsilon^{-d})$ samples to guarantee an error of $O(\epsilon)$ in situations where there is an important gap between the maximal eigenvalue and the remaining ones. 

A modified version of the problem lends itself to a better reduction complexity through this approach. Noticing that the right-hand side of \eqref{lambda_max as l_infinity} can be understood as taking an $\ell^\infty(S^d_1)$ norm, we suggest replacing \red $(\lambda_{max}(x, \theta))_+$ with $\left(\fint_{S^d_1} (z^\top  D^2_x j(\theta,x) z)^b_+ dz \right)^{1/b} $ for some power $b\geq 1$. Focusing on the case $b=1$, we substitute \eqref{eqn:SecondOrderVersion2} with 
\begin{equation}
  \label{eqn:SecondOrderVersion2Modified}
  \min_{\theta \in \Theta} \left\{  J(\mu_0, \theta) +  \frac{\delta^2}{2} \left(\E \left[ \left\{ \fint_{S^d_1} (z^\top  D^2_x j(\theta,x) z)_+ dz \right\}^{\tilde q} \right]\right)^{1/\tilde{q}} \right\}.
\end{equation}
\nc
Problem \eqref{eqn:SecondOrderVersion2Modified} can be seen as a \textit{different} form of curvature penalization where the stress is put on reducing the overall positive curvature of the second derivative. In this sense, this problem gives less importance to potential worst cases than \eqref{eqn:SecondOrderVersion2}. Conveniently, though, it is better suited for an approximation using randomized directions. 

\begin{proposition}
  Suppose Assumptions \ref{A2} and \ref{assump:Smoothness} hold. Let
  $\{z_i\}_{i=1,\ldots,m}$ be i.i.d. samples with $z_1\sim U(S^d_1)$ defined on a different probability space $(\Omega',\P', \F')$. Consider the problem
%   \begin{equation}
%     \begin{aligned}
%     \inf_{\theta \in \Theta } \quad & \left\{ \E_{x \sim \mu_0}\left[ \ell( X_{x,T},\theta_T)  \right] + \frac{\delta^2}{2 m} \sum_{i=1}^m  \red \left( \E_{x \sim \mu_0} \left[ (\rho^{z_i}\cdot z_i)_-^{\tilde q} \right]  \right)^{1/\tilde q} \nc  \right\} \\
%     \textrm{s.t.} \quad & 
%         \begin{cases} 
%           X,P \text{ as in \eqref{eqn:ControlExpanded}};\\
%           \beta \text{ as in Theorem \ref{Thm: Pontryagin}}; and    \\
%         d \gamma_{x,t}^{i}  =  D_{\xi }f\left( x_{x,t},\theta _{t}\right) \gamma _{x,t}^{i} dt,\\ %\quad \text{ for } i=1, \ldots m\\
%         d \rho^{i}_{x,t}= -D_{\xi }f\left( X_{x,t},\theta_s\right) \rho^v_{x,t} + \left\{ D_{\xi }^{2}f\left( X_{x,t},\theta_t\right) ^{\top}P_{x,t}\right\} ^{\top}\gamma_{x,t}^{i}dt , \\
%         \gamma_{x,T}^{i} = z_i., \qquad 
%         \rho^{z_i}_{x,T}= D_{\xi }^{2}\ell\left( X_{x,T}  ,\theta _{T }\right) \gamma _{x,T}^{i} \qquad \text{ for } i=1, \ldots m.\\
%     \end{cases}
%     \end{aligned}
%     \label{eqn:ControlExpandedSecondOrder2WithAdjoints}
%     \end{equation}
    
     \begin{equation}
    \begin{aligned}
    \inf_{\theta \in \Theta } \quad & \left\{ \E_{x \sim \mu_0}\left[ \ell( X_{x,T},\theta_T)  \right] +\red  \frac{\delta^2}{2} \left( \E_{x \sim \mu_0} \left[ \left\{ \frac{1}{m}\sum_{i=1}^m   (\rho^{z_i}_{x,0}\cdot z_i)_-   \right\}^{\tilde q} \right]   \right)^{1/\tilde q}  \nc  \right\} \nc \\
    \textrm{s.t.} \quad & 
        \begin{cases} 
          X,P \text{ as in \eqref{eqn:ControlExpanded}}; and    \\
        d \gamma_{x,t}^{i}  =  D_{\xi }f\left( x_{x,t},\theta _{t}\right) \gamma _{x,t}^{i} dt,\\ %\quad \text{ for } i=1, \ldots m\\
        d \rho^{z_i}_{x,t}= -D_{\xi }f\left( X_{x,t},\theta_{\red t \nc}\right)^{\red \top \nc} \rho_{x,t}^{\red z_i \nc} + \left\{ D_{\xi }^{2}f\left( X_{x,t},\theta_t\right) ^{\top}P_{x,t}\right\} ^{\top}\gamma_{x,t}^{i}dt , \\
        \gamma^{i}_{x,\red 0 \nc} = z_i, \qquad 
        \rho^{z_i}_{x,T}= D_{\xi }^{2}\ell\left( X_{x,T}  ,\theta _{T }\right) \gamma _{x,T}^{i} \qquad \text{ for } i=1, \ldots m.\\
    \end{cases}
    \end{aligned}
    \label{eqn:ControlExpandedSecondOrder2WithAdjoints}
    \end{equation}
  and let $V^{ \delta,m}$ be its optimal value. If $V^\delta$ is the optimal value in \eqref{eqn:SecondOrderVersion2Modified} then 
  \[  V^{\delta,m} \rightarrow V^\delta \qquad \P'-\text{a.s.}  \]
  and $V^{\delta,m} -  V^\delta$ satisfies a central limit theorem (so that the error is of order $m^{-1/2}$).
  \label{Prop:SecondOrder2WithAdjoints}
\end{proposition}

Problem \eqref{Prop:SecondOrder2WithAdjoints} also has an associated Pontryagin maximum principle as the other control problems discussed thus far, only that we do not write it out explicitly for brevity. We remark, however, that the corresponding Pontryagin principle can be used to motivate algorithms just as with the other problems discussed throughout this section. Notice that the number of state variables in problem \eqref{eqn:ControlExpandedSecondOrder2WithAdjoints} is $2m+3$. This control problem is more advantageous than one that directly tracks the Hessian matrix $D^2j(x,\theta)$ whenever $2m+3$ is considerably smaller than $d$.

\red
\begin{remark}  
  The variance of $V^{\delta,m} -  V^\delta$ in proposition \ref{Prop:SecondOrder2WithAdjoints} depends on the dimension $d$. 
  Consider, for example, the case $\tilde q=1$, which corresponds to the case $p=\infty$ in the discussion in section \ref{sec:SecondOrderRegResults}. First, notice that the terms $\rho^{z_i}\cdot z_i$ can be rewritten as $-z_i^\top D^2_x j(\theta, x) z_i$.
  Now, given that the matrix $D_x^2 j(\theta, x)$ is real and symmetric, and given that the $z_i$ are uniformly distributed on $S_1^d$, the spectral decomposition theorem can be used to find a linear growth of this variance with respect to the rank of the matrix, and \emph{a fortiori}, at most linear with respect to $d$. The linear in $d$ dependence of the variance estimate for fixed $x$ and $\theta$ can be upgraded to uniform ones over all $x$ and $\theta$ under smoothness assumptions on $j$. In terms of the sample size, this means that $m$ would need to grow a bit faster than linearly to guarantee convergence from the central limit theorem. Compare with the complexity of using and keeping track of the whole Hessian (as in \eqref{eqn:SecondOrderVersion2Modified}), which is quadratic with respect to the dimension.
  
  \red 
  Problem \eqref{eqn:SecondOrderVersion2Modified} can be easily extended to other values of $b \geq 1$, with $b \rightarrow \infty$ recovering the original problem \eqref{eqn:SecondOrderVersion2}. Notice, however, that the randomization strategy for the approximation of the integral degenerates for larger values of $b$. While problem \eqref{eqn:SecondOrderVersion2} is more closely connected to the original adversarial training, problem \eqref{eqn:SecondOrderVersion2Modified} is superior from a computational perspective.

\end{remark}
\nc

% As in the fist order case, we can also study the structure that problems  \eqref{eqn:SecondOrderVersion1} and  \eqref{eqn:SecondOrderVersion2} posses in the setting of control problems. Consider fist \eqref{eqn:SecondOrderVersion1}. Following...
% >>>>>>> 1e14820824ad590c21b861c5c3d18b9ce3527bc4

\subsection{Other forms of regularization in the literature}

In this section we review some works in the literature connected to robust training of neural networks and discuss four different ways of introducing regularization penalties. 

First, it is always possible to \textit{explicitly penalize the parameters} of a neural network. For example, in \cite{Ruthoto,EMaximPrinciple} the loss function $j(x,\theta)$ in \eqref{eqn:ControlBasic} is replaced with a loss function of the form:
\[ \tilde{j}(\theta,x )= \ell ( X_{x,T}, \theta_T) + \int_0^T \tilde{\Phi}(X_{x,t},\theta_t, \dot{\theta}_t) dt,  \]
where the running cost $\Phi$ explicitly penalizes the derivative of $\theta$ in time: this is one way to penalize parameters that is quite reasonable in the deep ResNet setting. A simple choice for $\tilde{\Phi}$ that allows us to draw a direct connection with the regularization problem \eqref{eqn:regularization} is 
\[ \tilde{\Phi}(\xi,\vartheta, u ):= \Phi(\xi,\vartheta)+ \lambda|u|^2,  \]
for $\lambda>0$. From a variational perspective, a regularization problem like the one described above admits optimal controls (parameters) in the classical sense. This contrasts with our problem \eqref{eqn:Regularization} where in general optimal solutions have to be sought in the space of \textit{generalized controls} a.k.a. \textit{Young measures}; see \cite{Pedregal}.

There are papers that \textit{explicitly penalize the input-to-output mappings} of a neural network. These are works that consider problems of the form:
\[ \inf_{\theta \in \Theta} \{ J(\mu_0,\theta) + R( g(X_{\cdot,T})) \}, \]
where $R(\cdot)$ is some type of regularization functional (typically a seminorm like for example the Lipschitz seminorm) acting explicitly on the input-to-output map $x\mapsto g(X_{x,T})$; here, the function $g$ is a simple function connected to the learning problem in hand (e.g. classification or regression) and that in general may depend on trainable parameters. The penalization of parameters via the regularity of their induced input-to-output maps has been considered in papers such as \cite{CalderOberman,NIPS2017_e077e1a5,weng2018evaluating} where it is has also been shown that the resulting function objectives do enforce adversarial robustness (i.e. stability to data perturbations). The paper \cite{ThorpeWang} considers this problem in the setting of graph-based learning.

Another approach to enforce robustness found in the literature is based on \textit{perturbation-based regularization terms}. These are approaches based on the construction of adversarial examples around the observed data that can be used to define a new ``perturbed" risk functional that is treated as regularizer. For example, \cite{Goodfellow1} considers the problem
\begin{equation}
\label{eqn:GoodfellowObjective}
   \inf_{\theta \in \Theta} \{ \E_{x \sim \mu_0}[ \tilde{j}(x,\theta) ] \}, \quad \text{ where }  \tilde{j}(x,\theta) :=  \alpha j(x,\theta) + (1-\alpha)j( x+ \delta \cdot \sign (\nabla_x j(x,\theta)),\theta), 
\end{equation}
for some $\delta>0$ and some $\alpha\in [0,1)$. In the above, $x+ \delta \cdot \sign (\nabla_x j(x,\theta))$ can be considered as an adversarial example. This specific way of constructing adversarial examples is known in the literature as the Fast Gradient Sign Method (FGSM) from \cite{Goodfellow1}. Although this method has some practical shortcomings, it has been found in \cite{wong_fast_2020} that it provides quickly state-of-the-art level robustification results when coupled with a random perturbation of the data $\mu_0$ within the $\ell^\infty$ ball considered.

Finally, objectives like the ones we presented in section \ref{sec:Regularization} have regularization terms that act explicitly on the derivatives of the \textit{loss function} (derivatives with respect to the input data) and thus penalize the parameters implicitly though the regularity that they induce on the loss function. In the literature, the work \cite{DoubleBackProp} proposed the use of a regularization term of the form
\[ \E_{x \sim \mu_0}[ \lVert \nabla_x j(x,\theta) \rVert ] \]
(i.e. as in problem \eqref{Robust control infty}), where it was also noticed that the resulting regularization problem could be implemented easily with a ``double backpropagation approach" (see our Remark \ref{rem:DoubleBackProp} below). There is a plethora of works that in the past few years have used and analyzed a similar \textit{input gradient regularization} approach (the case $p=\infty$ in our results or similar); see for example \cite{ObermanFinlay,GradientRegularization,SlavinRoss,StructuredGradReg,pmlr-v139-yeats21a}. More recently, higher order regularization terms penalizing \textit{curvature} of the loss function have been proposed, e.g. \cite{CurvatureRegularization}.

\subsection{Outline}

The rest of the paper is organized as follows. In section \ref{Sec:Regularized Robust Optimal} we present the proof of Theorem \ref{thm:Regularization} using two approaches, including a formal one based on the geometric structure of optimal transport. In section \ref{sec:PerturbationRobustAlgorithm} we discuss connections between regularization problems and perturbation-based training algorithms. In section \ref{Sec: Locally robust problem} we discuss the Pontryagin principle associated to the optimal control formulation of the first order regularization problem \eqref{eqn:Regularization} (i.e. Theorem \ref{Thm: Pontryagin}). Section \ref{sec:SecondOrderReg} is devoted to second order regularization problems and in particular the motivation of problems \eqref{eqn:SecondOrderVersion1} and \eqref{eqn:SecondOrderVersion2}, as well as their optimal control reformulations. Using the Pontryagin principles discussed throughout the paper we motivate a family of algorithms for the training of robust neural networks in section \ref{sec:algorithm}, \red and in section \ref{sec:Numerics} we present a series of numerical experiments to illustrate the performance of the algorithms. \nc We wrap up the paper in section \ref{sec:Conclusions} where we present some conclusions.

% \subsection{Notation}

% \begin{itemize}
%     \item $x, \tilde x$ represent vectors in $\R^d$.
% \item$y$ represents ``labels"; either real numbers (for regression) or elements in a finite set (for classification).

% \item$z$ is the state variable in our generic control problem. We will take it to be either $z=(x,y)$ or $z=(x,\tilde x , y)$.

% \item $\theta$ represents the weights of the ``ideal" (continuous time) ResNet neural network. 

% \item $\beta \in\R^d$ and represents the vector used in the last layer of the neural network for either classification or regression.  

% \item$\mu, \tilde \mu$ to denote probability measures on $\R^d$.  

% \item Generically $\nu$ will represent a probability measure over the variable $z$. 
% \end{itemize}

% \begin{proposition}
%   Suppose Assumption \ref{A1} is in force, and that the problem \eqref{Robust problem 2} has a (local) solution $(\theta^*, X^*)$. Then, there exists $p:[0,T]\rightarrow \R^n$ satisfying

% \end{proposition}

% \begin{proposition}

%   Assume further that $0<\delta \ll 1$, and consider the problem 

%   \begin{equation}
%     \inf\limits_{\theta \in \Theta, \beta \in \R^d } \inf_{a\geq 0} \{ \tilde J (\theta, \beta, \mu_0, a)     \},
%     \label{Robust problem small delta}
%   \end{equation}

%   Then, there exists a solution ..... and ....
% \end{proposition}

\section{First order regularization}
\label{Sec:Regularized Robust Optimal}

\subsection{Proof of Theorem \ref{thm:Regularization}}

In this section we present the proof of Theorem \ref{thm:Regularization}, first rigorously, and then by providing a formal argument that relies on the geometric structure of the space of probability measures $\mathcal{P}_p(\R^d)$ endowed with the $W_p$ distance. For our rigurous proof we use the fact that for every $\theta \in \Theta$ we have:
\begin{equation}
     \max_{\tilde \mu : \: W_p(\mu_0, \tilde \mu) \leq \delta} J(\tilde \mu, \theta)  = \max_{\pi \in \mathcal{F}_{\mu_0, \delta} }  \left\{ \int_{\R^d \times \R^d} j(\tilde x , \theta) d \pi(x,\tilde x)  \right\},
     \label{eqn:OptimiIntermsCouplings}
\end{equation}
where $\mathcal{F}_{\mu_0, \delta}$ is the set of all Borel probability measures $\pi$ on $\R^d \times \R^d$ whose first marginal is $\mu_0$ and satisfy
\[\int_{\R^d \times \R^d} \lVert x-\tilde x \rVert^p d \pi(x,\tilde x) \leq \delta^{\red p \nc}; \]
see Remark \ref{rem:CouplingsRem}. Identity \eqref{eqn:OptimiIntermsCouplings} has been used repeatedly in the literature of distributionally robust optimization to obtain dual representations for generic adversarial problems of the form \eqref{Robust problem} and to in turn propose new frameworks for statistical inference (e.g. \cite{RobustSolutionsBenTal,blanchet2019robust,Blanchet2,OPT-026,Kuhn,Wiesemann2014DistributionallyRC}; see section \ref{sec:PerturbationRobustAlgorithm} below for an explicit formula for the dual problem). In this section, however, we do not focus on the dual representation associated to the adversarial problem, but rather, on the flat geometry of the set of $\pi$s that parameterize the right hand side of \eqref{eqn:OptimiIntermsCouplings}.
\nc

\begin{proof}[Proof of Theorem \ref{thm:Regularization}]
We follow an approach based on lifting the probability distributions $\pi$ to a space of random variables as used, for example, when defining and studying the L-derivative (see for example sections 5.1-5.2 in \cite{carmona2018probabilistic}). The main idea is to use the fact that, over an atomless probability space $(\Omega,\F,\P)$, and for any coupling $\pi \in \Gamma(\mu_0, \tilde \mu)$ where $\mu_0, \tilde \mu$ are probability distributions on $\R^d$, we can construct $\R^d$-valued random variables $X,\tilde X$ on $\Omega$ with joint distribution $\pi$ (for a proof see Proposition 9.1.2 and Theorem 1.13.1 in \cite{dudley2018real}). We can then operate using these random variables to approximate our optimization problem. Since the development only makes use of our assumptions and the properties of $\pi$, the result is independent of the chosen probability space.

With this idea in mind, using Assumptions \ref{A1}, we can write for any coupling $\pi \in \Gamma(\mu_0,\tilde \mu)$ that
\begin{align*} 
  J(\tilde \mu, \theta) & = \E[j(\tilde X,\theta)] \\
  & = \E \left[ j(X,\theta) + \nabla_x j(X,\theta)\cdot (\tilde X-X) \right .\\
  & \qquad + \left. \int_0^1  \{ \nabla_x j(X + \lambda (\tilde X-X),\theta) - \nabla_x j(X,\theta) \}\cdot (\tilde X-X) \dd \lambda \right],
\end{align*}
where $X,\tilde X$ are random variables with $(X,\tilde X) \sim \pi$. Using now the fact that $\nabla_x j( \cdot, \theta)$ is Lipschitz, it follows that
\begin{align*} 
  \left| J(\tilde \mu, \theta) -  \E \left[ j(X,\theta) + \nabla_x j(X,\theta)\cdot (\tilde X-X) \right] \right|
%   \leq \E \left[ j(X,\theta) + \nabla_x j(X,\theta)\cdot (\tilde X-X) \right .\\
%   & \qquad + \left. \int_0^1 \lambda \{ |\nabla_x j |_{Lip} |\tilde X-X|^2 \right]\\
   \leq \frac{1}{2} Lip( \nabla_x j(\cdot, \theta) ) \E [\lVert\tilde X-X\rVert_e^2],
\end{align*}
where $\lVert \cdot \rVert_e$ is the Euclidean norm. 

Note that the constraint $\int_{\R^d \times \R^d} \lVert x-\tilde x \rVert^p d \pi(x,\tilde x) \leq \delta$ means that we only consider couplings where 
\[ \E [\rVert\tilde X-X \rVert_e^2] \leq C \E[ \|\tilde X-X\|^2] \leq C (\E[ \|\tilde X-X\|^p])^{\red 2/p \nc} \leq C\delta^2   \]
where $C$ is a constant appearing due to the equivalence of norms in $\R^d$; in the last line we use the fact that $p\geq 2$ to apply Jensen's inequality. Since the above is a uniform control in the space of feasible solutions $\pi$ and $\Theta$ (thanks to Assumption \ref{assump:Smoothness}), we conclude that
\[ \sup_{\tilde \mu: \: W_p(\mu_0,\tilde \mu)\leq \delta } J(\tilde \mu,\theta) = \sup_{ \pi \: : \int_{\R^d \times \R^d} \lVert x-\tilde x \rVert^p d \pi(x,\tilde x) \leq \delta^{\red p \nc}}  \{J(\mu_0,\theta) + \E[\nabla_x j(X,\theta)\cdot (\tilde X-X) ] \} +  O(\delta^2), \]
for $O(\delta^2)$ independent of $\theta$. Clearly, the first term within the supremum on the right-hand side is independent of $\tilde \mu$. We then only need to optimize the second term. Now, by a (generalized) H\"older inequality (see for example \cite{yang1991generalized}), it follows that
\[ \E[\nabla_x j(X,\theta)\cdot (\tilde X-X) ] \leq \left(\E[\|\nabla_x j(X,\theta)\|_*^q]\right)^{1/q} \left(\E[\|\tilde X-X) \|^p]\right)^{1/p} \]
where equality can be attained whenever $1\leq p\leq \infty$ for a proper choice of $\tilde X-X$ with $\left(\E[\|\tilde X-X \|^p]\right)^{1/p} = \delta$. Therefore,
\[ \sup_{\tilde \mu: \: W_p(\mu_0,\tilde \mu)\leq \delta } J(\tilde \mu,\theta) = J(\mu_0, \theta) + \delta \left(\E[\|\nabla_x j(X,\theta)\|_*^q]\right)^{1/q}, \]
which in turn deduces our claim since the $O(\delta^2)$ term is uniform over all $\theta \in \Theta$.
\end{proof}
\nc
\begin{remark} Coming back to the literature on L-derivative mentioned in the proof of Theorem \ref{thm:Regularization}, we remark that in this case $\partial_\mu J(\mu_0,\theta) = \nabla_x j(X,\theta)$ (see section 5.2.2, Example 1 in \cite{carmona2018probabilistic} ). As pointed out before, this quantity does not depend on the choice of lifting.  
\end{remark}

\red

\grn

% \begin{lemma}
% Let $f,g:\R^d\rightarrow \R$ such that $g$ is convex and $f$ is continuously differentiable near a point  $x^*\in \R^d$. Then, if $x^*$ is a local minimum of $f+g$ then
% \[- \nabla f(x^*) \in \partial g(x^*).\]
% \end{lemma}

% \begin{proof}
%   A short proof uses \emph{generalized gradients} (see Chapter 10 in \cite{clarke2013functional}). 

%   By Fermat's rule, the sum rule and Theorem 10.8 in \cite{clarke2013functional} we get
%   \[ 0 \in  \partial_C (f(x^*)+g(x^*))  \subset  \partial_C f(x^*)+ \partial_C g(x^*) = \{\nabla f(x^*)\} + \partial g(x^*) \]
%   from where the claim follows.
% \end{proof}
\nc

\subsubsection{A formal geometric analysis in the space $\mathcal{P}_p(\R^d)$}
\label{sec:TaylorGeometry}

In this section we present an alternative formal geometric analysis for Theorem \ref{thm:Regularization}. We restrict to the case where $p \geq 2$ and for simplicity assume that $\lVert \cdot\rVert$ is the Euclidean norm. The idea in this geometric approach is to use the formal differential structure of the space of probability measures $\mathcal{P}_p(\R^d)$ endowed with the $W_p$ distance and carry out a Taylor expansion of the function $\tilde \mu \mapsto J(\tilde \mu, \theta)$ around $\mu_0$. Ultimately, the goal is to replace the function $J(\tilde \mu, \theta)$ with an approximation written in terms of a natural retraction of $\tilde \mu$ to the tangent space of $\mathcal{P}_p(\R^d)$ at the point $\mu_0$. For a discussion on the general differential geometric perspective in the Wasserstein space (i.e. when $p=2$) see Chapter 8.2 in \cite{VillaniBook}, and for more details on the geometry of the space $(\mathcal{P}_p(\R^d), W_p)$ see Chapter 8 in \cite{AmbrosioGigliSavare}.

In more precise terms, let $\tilde \mu$ belong to the  $W_p$-ball of radius $\delta$ around $\mu_0$. For every such $\tilde \mu$ we consider the constant speed geodesic $t \in [0,1] \mapsto (\mu_t, V_t)$ that connects $\mu_0$ with $\tilde \mu$. Namely, the continuity equation
\begin{equation*}
\partial_t \mu_t + \divergence(V_t \mu_t) =0,\quad t\in (0,1).
\end{equation*}
and the relations
\[ \int_{0}^t \int_{\R^d} \lVert V_s(x) \rVert^p d\mu_s(x) ds =   t (W_p(\mu_0, \tilde \mu))^p, \quad t \in [0,1], \]
are satisfied. Here, $V_t$ is a vector field in $\R^d$ and $\mu_t $ is a probability measure in $\R^d$ which at time $t=1$ coincides with $\tilde \mu$. The pair $(\mu_t, V_t)$ can be characterized further. Indeed, assuming that there exists an \textit{optimal} transport map $T^*$ between $\mu_0$ and $\tilde \mu$ for the $c_p$ cost (for example assuming that $\mu_0$ has a density with respect to the Lebesgeue measure) we can write:
\begin{equation}
\label{eqn:Geodesic}
  \begin{cases}  \mu_t=  T_{t\sharp } \mu_0, \quad \forall t \in [0,1], \\
V_t( T_t(x) ) = T^*(x)-x, \quad \forall x \in \text{supp}(\mu_0), \forall t \in [0,1],
\end{cases}
\end{equation}
where the map $T_t$ is given by 
\[ T_t(x):= t T^*(x) +(1-t) x.\]
The function $\tilde \mu \mapsto V_0 \in L^p(\R^d : \R^d, \mu_0)$ can be understood as a \textit{logarithmic map}, i.e. a map that in particular sends points in the manifold $\mathcal{P}_p(\R^d)$ to tangent vectors at $\mu_0$ and that satisfies:
\[ \lVert V_0 \rVert^p_{L^p(\R^d: \R^d, \mu_0)} = \int_{\R^d}  \lVert V_0(x) \rVert^p d\mu_0(x) = (W_p(\mu_0, \tilde \mu) )^p. \]

Let us now find an approximation for the function $J(\tilde \mu, \theta)$ in terms of an expression involving $V_0$. For that purpose we Taylor-expand the function $j(\cdot ,\theta)$ along the geodesic \eqref{eqn:Geodesic}. First, following equation 8.1.4. in \cite{AmbrosioGigliSavare} we obtain:
\[ \frac{d}{dt} J(\theta, \mu_t) = \frac{d}{dt} \int_{\R^d} j(x,\theta) d\mu_t(x) = \int_{\R^d} \nabla_x j(x,\theta)  \cdot V_t(x) d\mu_t(x); \]
notice that, geometrically speaking, when $p=2$ the above equation is the standard relation between directional derivatives and gradients in the Wasserstein space. We can also compute the second derivative along the geodesic as follows:
\begin{align}
\label{eqn:SecondDerivative}
\begin{split}
\frac{d^2}{dt^2} J(\theta, \mu_t) & =  \frac{d}{dt} \left( \int_{\R^d} \nabla_x j(x,\theta) \cdot V_t(x) d\mu_t(x) \right)  
\\& =  \frac{d}{dt} \left( \int_{\R^d} \nabla_x j(\theta, T_t(x) ) \cdot V_t(T_t(x)) d\mu_0(x) \right)
\\& = \int_{\R^d}  (D^2_xj(\theta, T_t(x) ) V_0(x) )\cdot V_0(x) d \mu_0(x).
\end{split}
\end{align}
In particular, 
\[ \left| \frac{d^2}{dt^2} J(\theta, \mu_t)  \right|  \leq C \int_{\R^d} \lVert V_0(x) \rVert^2 d \mu_0(x) \leq C \left( \int_{\R^d} \lVert V_0(x) \rVert^p d \mu_0(x)\right)^{2/p} \leq C \delta^2, \]
where the constant $C$ is uniform over all $t \in [0,1]$, all $\theta \in \Theta$ (thanks to Assumption \ref{assump:Smoothness}) and all $\tilde \mu $ with $W_p(\mu_0, \tilde \mu) \leq \delta$; notice that in the second to last line we have used Jensen's inequality since $p\geq 2$.

From the above computations we obtain:
\begin{align*}
\begin{split}
J(\tilde \mu, \theta) &= J(\mu_0, \theta) +  \frac{d}{dt} J(\theta, \mu_t) \big|_{t=0}  + O(\delta^2)
\\& =  J(\mu_0, \theta) +  \int_{\R^d} \nabla_x j(x,\theta)  \cdot V_0(x) d\mu_0(x)   + O(\delta^2),
\end{split}
\end{align*}
where again we notice that the $O(\delta^2)$ term is uniform over all $\theta \in \Theta$ and all $\tilde{\mu}$ within $W_p$-distance $\delta$ from $\mu_0$. Using the relation between the $\tilde \mu$s and their corresponding $V_0$s we can thus expect that up to an error of order $O(\delta^2)$ (independent of $\theta \in \Theta$), the expression:
\[ \max_{\tilde \mu: \: W_p(\mu_0,\tilde \mu)\leq \delta} J(\theta,\tilde \mu ) \]
(an optimization problem over a curved manifold) is equal to:
\begin{equation}
\label{eqn:FirstOrderFlat}
    J(\mu_0, \theta) +  \max_{ V_0 : \:  \lVert V_0\rVert_{L^p(\R^d:\R^d, \mu_0 )} \leq \delta }   \left\{ \int_{\R^d} \nabla_x j(x,\theta) \cdot V_0(x) d\mu_0(x)   \right\},
\end{equation}
which is an optimization over a flat Banach space. Since \eqref{eqn:FirstOrderFlat} is simply a dual representation for the $L^q(\R^d: \R^d, \mu_0) $-norm of $\nabla_x j(\cdot ,\theta)$, it follows that \eqref{eqn:FirstOrderFlat} is equal to:
\[  J(\mu_0, \theta) + \delta \left(  \int_{\R^d} \lVert  \nabla_x j(x, \theta)\rVert^q d\mu_0(x) \right)^{1/q},    \]
which is the objective function in \eqref{eqn:Regularization}. Notice that the $V_0$ achieving the maximum takes the form:
\begin{equation}
\label{eqn:V0Max}
V_0(x):= \delta \left[ \int_{\R^d} \lVert  \nabla_x j(\tilde x, \theta)\rVert^q d\mu_0(\tilde x)  \right]^{-1/p}  \frac{\nabla_x  j(x,\theta)}{ \| \nabla_x  j(x,\theta) \|^{1-1/(p-1)}  },
\end{equation}
with the convention $\mathbf{0}/0=\mathbf{0}$. %provided that $\lVert \nabla_x j(x,\theta)  \rVert_{L^q(\R^d: \R^d, \mu_0)} \not =0 $; we can take any $V_0$ otherwise.

\begin{remark}
\label{rem:OTapproachFirstOrder}
There are a few steps in the above analysis that would need further justification in order to make this analysis into a rigurous proof. Here we offer some comments on this direction.

\begin{enumerate}
\item First, we have used the existence of optimal transport maps to write the geodesic \eqref{eqn:Geodesic} and to define an associated retraction map $\tilde \mu \mapsto V_0$. This can be done, for example, if we assume that $\mu_0$ is absolutely continuous with respect to the Lebesgue measure. The reduction to the absolutely continuous case can be accomplished by an approximation argument since the $O(\delta^2)$ correction terms in the above analysis only depend on the control we have on the Hessian (in $x$) of the loss function given by the Assumption \ref{assump:Smoothness}.

\item Going from the maximization problem $\max_{\tilde \mu: \: W_p(\mu_0,\tilde \mu)\leq \delta} \{ J(\theta,\tilde \mu ) \}$ to \eqref{eqn:FirstOrderFlat} is motivated by the fact that on a finite dimensional smooth manifold one can find a one to one correspondence between points in a geodesic ball with small enough radius $R$ (in particular smaller than the injectivity radius of the manifold) and tangent vectors at the center of the ball that have norm less than $R$. In the space $\mathcal{P}_p(\R^d)$, however, this intuition breaks down. To illustrate how one can still recover \eqref{eqn:FirstOrderFlat} using optimal transport theory let us consider the case $p=2$ for concreteness and assume that $\mu_0$ is absolutely continuous with respect to the Lebesgue measure. In that case, the $V_0$s induced by $\tilde \mu$s within $W_2$-distance $\delta$ from $\mu_0$ can be written as:
\[ \delta( \nabla_x \varphi(x) - \frac{x}{\delta} ) = \delta \nabla_x ( \varphi(x) - \frac{\lVert x\rVert ^2}{2\delta})   \]
for $\varphi $ a convex function, as it follows from Brenier's theorem (see Theorem 2.12 in \cite{VillaniBook}). Now, by Assumption \ref{assump:Smoothness}, the function 
\[ x \mapsto \frac{j(x,\theta)}{c} + \frac{\lVert x \rVert^2}{2 \delta}\] 
is a convex function for all small enough $\delta$. In the above, $c = \lVert \nabla_x j(\cdot ,\theta) \rVert_{L^2(\R^d: \R^d, \mu_0)}$. We can thus take $\varphi(x) =  \frac{j(x,\theta)}{c} + \frac{\lVert x \rVert^2}{2 \delta}$ (assuming $\delta$ is small enough) and $V_0(x) = \delta( \nabla \varphi(x) - x)$. It is clear that this $V_0$ maximizes \eqref{eqn:FirstOrderFlat} when $p=2$.

\item It is worth mentioning that the formal analysis presented here does not use specific attributes of the Euclidean norm, and in fact can be used for general norms (as done in our rigurous proof of Theorem \ref{thm:Regularization}) with the difference that the form of the maximizer in problem \eqref{eqn:FirstOrderFlat} would be in general more cumbersome. One notable exception is the case of the $\ell^\infty$ norm on $\R^d$-vectors and the induced $L^\infty$ norm on vector fields (i.e. $p=\infty$). Indeed, in that case the maximizer takes the form:
\[ V_0(x)= \delta \text{sign}(\nabla_x j(x,\theta  )),  \]
where the $\text{sign}$ function acts coordinatewise on vectors.

\item Finally, it is worth mentioning that one of the main motivations for presenting this alternate analysis is to introduce additional tools that may come in useful when considering different adversarial learning problems where for example the energy to maximize is not simply an integral with respect to a measure (here $J(\tilde \mu, \theta)$) but rather an energy which in general may include entropic or interaction terms as done recently in \cite{DongDP0020}. Most modern algorithms used for training robust neural networks include a step where data points are randomly perturbed before considering any drift information induced by the loss function. We believe that making a more concrete connection between said algorithms and a distributionally robust optimization problem would require analyzing this type of entropic term. This is work that is left for the future.

\end{enumerate}

\end{remark}

\nc

\nc

\subsection{Connection with perturbation-based training algorithms}
\label{sec:PerturbationRobustAlgorithm}

    The current literature on robust training is dominated by algorithms based on constructing explicit adversarial samples around a given distribution that are then used as \emph{the} training samples for the network. See for example \cite{carlini_evaluating_2019,carlini_evaluating_2017,madry_deep_2019,tramer_ensemble_2020}, and references therein.

    We can relate the results of our study in terms of adversarial samples. Indeed, equation \eqref{eqn:Geodesic} suggests considering a transport map of the form:
    \[ \hat{T}(x) := V_0(x) + x, \]
    for the $V_0$ maximizing \eqref{eqn:FirstOrderFlat}, and in turn consider the associated measure $\hat{\tilde{\mu}}:= \hat{T}_{\sharp} \mu_0$ (the pushforward of $\mu_0$ by $\hat{T}$); notice that $\hat{\tilde \mu}$ depends on $\theta$, and for each $\theta$ $\hat{\tilde \mu}$ is a natural surrogate for the maximizer of problem \eqref{thm:Regularization}. One may then consider an objective function of the form:
    \[ \theta \in \Theta \mapsto J(\hat{\tilde \mu}, \theta) \]
    or more generally
    \[\theta \in \Theta \mapsto \alpha J(\mu_0, \theta) + (1-\alpha) J(\hat{\tilde \mu}, \theta),  \]
    for some $\alpha \in [0,1)$, and the corresponding minimization problem over $\theta \in \Theta$ in order to enforce robustness. It is straightforward to show that, under the assumptions of Theorem \ref{thm:Regularization}, the minimum value of the resulting problem when $\alpha= 1-\delta$ is within an error of order $\delta^2$ of the minimum value of the regularization problem \eqref{eqn:Regularization}. We remark that if $V_0$ is taken to be $V_0(x):= \delta \cdot \text{sign}(\nabla_x j(x,\theta  ))$, i.e. as in iii) in Remark \ref{rem:OTapproachFirstOrder}, then one can recover problem \eqref{eqn:GoodfellowObjective} introduced in \cite{Goodfellow1}.

%     It is straightforward to show that under the assumptions of Theorem \ref{thm:Regularization} we have that 
%     \[ j(x + \Delta x, \theta) = j(x, \theta) + \Delta x \cdot \nabla_\xi j(x, \theta)  + o(|\Delta x|).\]
%     Now, set
%     \[\Delta x := \delta  |\nabla_\xi j(x,\theta)|^{q-1} (r\E_{x_0\sim \mu_0}[|\nabla_\xi j(x_0,\theta)|^q])^{1/r}  \nabla_\xi j(x,\theta), \]
%     and define ${\tilde U}^*_\delta$ the solution to
%     \begin{equation}
%       \label{eqn:Regularization2}
%       \inf_{\theta \in \Theta} \left\{ J(\theta, \tilde \mu_0) \right\}.
%     \end{equation}
%     where $\tilde \mu_0 = \mu_0 \circ f_\delta^{-1}$, the pushforward of $\mu_0$ using
%     \[f_\delta(x):= x+ \delta  |\nabla_\xi j(x,\theta)|^{q-1} (r\E_{x_0\sim \mu_0}[|\nabla_\xi j(x_0,\theta)|^q])^{1/r}  \nabla_\xi j(x,\theta) ].\]
%     Then, it follows that ${\tilde U}^*_\delta-U^*_\delta = o(\delta)$, and a similar result to  Theorem \ref{thm:Regularization} applies. Note that this means that the problem can be understood, up to a small parameter depending on $\delta$, as perturbing the sample via the function $f_\delta.$ 

% \red

To conclude our discussion in this section, let us point out that maximizers of problem \eqref{eqn:RobustObjective} have a clear mathematical characterization. Unfortunately, even with this characterization an optimal $\tilde \mu $ is in general difficult to compute explicitly.

\begin{proposition}
Let $\delta>0$ and let $\theta \in \Theta$. Then
\begin{equation}
    \max_{\pi \in \mathcal{F}_{\mu_0, \delta} }  \left\{ \int_{\R^d \times \R^d} j(\tilde x , \theta) d \pi(x,\tilde x)  \right\}  = \min_{\gamma \geq 0} \ \left\{  \gamma \delta^{\red p \nc} +  \E_{x\sim \mu_0} \left[ j^\gamma(x, \theta) \right] \right\},
    \label{eqn:DualAL}
\end{equation}
where 
\[ j^{\gamma}(x, \theta) := \sup_{\tilde x } \left\{ j(\tilde x , \theta ) -  \gamma \lVert x- \tilde x \rVert^{\red p \nc} \right\}. \]

Moreover, if $\pi^*$ is a solution to the problem on the left hand side, then there exists a solution to the problem on the right hand side $\gamma^*$ such that for $\pi^*$-a.e. $(x, x^*)$:
\[ x^* \in \argmax_{\tilde x \in \R^d} \left\{ j(\tilde x , \theta)  - \gamma^* \lVert x- \tilde x\rVert^{\red p \nc} \right\}.  \]
Finally, the second marginal of $\pi^*$ is a solution to the problem \eqref{eqn:RobustObjective}.

\end{proposition}

\begin{proof}
 The dual characterization \eqref{eqn:DualAL} is discussed, for example, in any of the following references: \cite{RobustSolutionsBenTal,blanchet2019robust,Blanchet2,OPT-026,Kuhn,Wiesemann2014DistributionallyRC}. The characterization for the minimizers is a straightforward consequence of the zero duality gap.
\end{proof}
\nc

\subsection{Pontryagin principle for the first-order regularized robust control problem}
\label{Sec: Locally robust problem}
  We start by writing the modified control problem in a slightly different form. Under Assumption \ref{A1} and fixed but arbitrary control $\theta$  it follows that 
  \[ \nabla_x j(x,\theta) = -P_{x,0}. \]
  Indeed, let $\Delta x$ be an arbitrary unitary vector in $\R^d$.  From the flow property of the ordinary differential equation for $X$ we get that for almost all $t\in [0,T]$
  \[ \lim_{\epsilon \downarrow 0 } \frac{1}{\epsilon} \{X_{x+\epsilon \Delta x, t} - X_{x,t}\}  = \zeta^{\Delta x}_t    \]
  where $\zeta^{\Delta x}$ satisfies
  \begin{equation}
    \begin{cases}
      \dot{\zeta}^{\Delta x}_t = D_\xi f(X_{x,t},\theta_t) \zeta^{\Delta x}_t  \\
      \zeta^{\Delta x}_0 = \Delta x.
    \end{cases}  
  \label{Directional_derivative_X}    
  \end{equation}
  
  Thus, the chain rule implies that
  \begin{align*}
    \nabla_x j(x,\theta) \cdot \Delta x & 
    % =  \lim_{\epsilon \downarrow 0 } j(X_{x+\epsilon \Delta x,T},\theta) - j(X_{x,T},\theta)    
    = \nabla_\xi \ell (X_{x,T},\theta_T) \cdot \zeta^{\Delta x}_T + \int_0^T  \nabla_\xi \Phi(X_{x,s}, \theta_s) \cdot \zeta^{\Delta x}_s \dd s\\
    & = - P_{x,T} \cdot \zeta^{\Delta x}_T + \int_0^T  \nabla_\xi \Phi(X_{x,s}, \theta_s) \cdot \zeta^{\Delta x}_s \dd s\\
    & = - P_{x,0} \cdot \zeta^{\Delta x}_0 - \int_0^T \dot{P}_{x,s} \cdot \zeta^{\Delta x}_s \dd s - \int_0^T P_{x,s} \cdot \dot{\zeta}^{\Delta x}_s \dd s +\int_0^T  \nabla_\xi \Phi(X_{x,s}, \theta_s) \cdot \zeta^{\Delta x}_s \dd s\\
    & = -P_{x,0} \cdot \zeta^{\Delta x}_0 = -P_{x,0} \cdot \Delta x.
  \end{align*} 
  This equivalence between the dual variable $P$ and (minus) the derivative of the loss function with respect to the input allows us to rewrite problem \eqref{eqn:Regularization} in the form \eqref{eqn:ControlExpanded}. We are ready to provide a proof of the Pontryagin principle result stated in Theorem \ref{Thm: Pontryagin}.

  \begin{proof}[Proof of Theorem \ref{Thm: Pontryagin}]
  
    The proof follows the well-known ``needle'' perturbation approach: we take the optimal control and change it in a small interval; then, we deduce the effect on the overall value function and deduce first order conditions of optimality from linear expansions and integration by parts.
  
    To simplify the problem, note that we can and will assume, without loss of generality, that there is no running cost (i.e., $\Phi\equiv 0$) by transforming the running cost in a state variable: indeed, set $\hat x:= (x,x')$ with $\hat x_0 = (x_0,0)$, $\hat \p := (\p,-1)$,  $\hat f:= (f, \Phi)$, and $\hat \ell(\bar x) := \ell(x) + x' $, and note that the reduced is an equivalent control problem, still satisfies Assumptions \ref{A2}, and does not contain any running cost. 

    Now, let $\tau \in (0,T)$ be a Lebesgue point for $(f(X^*_t,\theta^*_t), D_\xi f(X^*_t,\theta^*_t) )$. By Assumption \ref{A2}, the set of such points is dense in $[0,T]$. For $\epsilon \in (0, T-\tau) $ and $\eta \in \Theta$ let
    \[ \theta_t^{\epsilon, \tau} = 
    \begin{cases}
      \eta & \text{if } t\in [\tau-\epsilon,\tau] \\
      \theta^*_t & \text{otherwise}   
    \end{cases},  \]
    and let $X_t^{\epsilon,\tau}, P_t^{\epsilon,\tau}$ be the solutions of    
    \begin{align*}
      X_{x,t}^{\epsilon,\tau} & = x + \int_0^t f( X_{x,s}^{\epsilon,\tau}, \theta^{\epsilon,\tau}_s) \dd s \\%\label{X dynamics}\\
      P_{x,t}^{\epsilon,\tau} & = -\nabla_\xi \ell(X_{x,T}^{\epsilon,\tau},\theta_T^*) + \int_t^T \{ D_\xi f( X_{x,s}^{\epsilon,\tau}, \theta^{\epsilon,\tau}_s)^\top P_{x,s}^{\red \epsilon, \tau\nc} \} \dd s %\label{P dynamics}.
    \end{align*}
    that is, solutions of the forward variable using the control  $\theta^{\epsilon, \tau}$ instead of the optimal $\theta^*$.

    Let us study the $\epsilon$-order effect of this change of control policy. Let
    \[
    \begin{cases}
      \dot{u}_{x,t} = D_\xi f(X^*_{x,t},\theta^*_t) u_{x,t} & \text{for } t>\tau; \\
      u_{x,t} = f(X^{*}_{x,\tau}, \eta ) -  f(X^{*}_{x,\tau}, \theta^*_\tau )  & \text{for } t =\tau;\\      
      u_{x,t} = 0 & \text{for } t <\tau.
    \end{cases}      
    \]
    and
    \[
    \begin{cases}      
      \begin{split}
     \dot{v}_{x,t} = & - D_\xi f(X^*_{x,t},\theta^*_t)^\top v_{x,t}  - (D^2_\xi f(X^*_{x,t},\theta^*_t) u_{x,t})^\top P^*_{x,t} \\
        & + \delta_{\tau} \{ D_\xi f(X_{x,\tau}^*, \theta^*_\tau) - D_\xi f(X_{x,\tau}^*, \eta)   \}^\top P_{x,\tau}^* 
      \end{split}  \\
      v_{x,T} = - D^2_\xi \ell(X^*_{x,T}) u_{x,T};
   \end{cases}
   \]
   where the term $ \delta_{\tau} \{ D_\xi f(X_{x,\tau}^*, \theta^*_\tau ) - D_\xi f(X_{x,\tau}^*, \eta)   \}^\top P_\tau^*  $ denotes the jump arising at time $\tau$ due to the change in control. We can show (as in \cite{yong1999stochastic}) that
    \begin{align*}  
      \lim_{\epsilon \downarrow 0} \frac{1}{\epsilon} ( X^{\epsilon,\tau}_{x,t} - X^{*}_{x,t} )  = u_{x,t};\\
      \lim_{\epsilon \downarrow 0} \frac{1}{\epsilon} ( P^{\epsilon,\tau}_{x,t} - P^{*}_{x,t} ) = v_{x,t}.
    \end{align*}
    Since optimality implies
    \[  \E_{x \sim \mu_0}  \left[ \ell(X^{\epsilon,\tau}_{x,T})\right] +  \red \delta   \nc \left(\E_{x \sim \mu_0} \left[ \lVert P^{\epsilon,\tau}_{x,0}\rVert ^q  \right]\right)^{1/q}
    %+ \sqrt{\delta \langle |P_0^{\epsilon,\tau}(.)|^2,\mu_0 \rangle} 
    \geq \E_{x \sim \mu_0}  \left[ \ell(X^{*}_{x,T})\right] +  \red\delta  \nc \left(\E_{x \sim \mu_0} \left[ \lVert P^{*}_{x,0}\rVert^q  \right]\right)^{1/q}
    %\left\langle  \ell(X_T^*(.)) ,\mu_0\right\rangle  +
    % \sqrt{\delta \langle |P_0^*(.)|^2,\mu_0 \rangle }
       \]
    we have from Assumption \ref{A2}  and dominated convergence that
    \begin{align}  
      0 & \leq  \lim\limits_{\epsilon\downarrow 0} \frac{1}{\epsilon} \left\{  \E_{x \sim \mu_0}  \left[ \ell(X^{\epsilon,\tau}_{x,t}) - \ell(X^*_{x,t}) \right] + \red \delta\nc ( \left(\E_{x \sim \mu_0} \left[ \lVert P_0^{\epsilon,\tau}(x_0)\rVert^q  \right]\right)^{1/q} -  \left(\E_{x \sim \mu_0} \left[ \lVert P^*_{x,0}\rVert^q  \right]\right)^{1/q} ) \right\} \notag\\
      %\left\langle \ell(X^{\epsilon,\tau}_t(.)) - \ell(X^*_t(.)), \mu_0\right\rangle 
      %+ \sqrt{\delta \langle |P_0^{\epsilon,\tau}(.)|^2,\mu_0 \rangle} 
      %- \sqrt{\delta \langle |P_0^*(.)|^2,\mu_0 \rangle }  \right\} \notag\\
      & =   \E_{x \sim \mu_0}  \left[ \nabla_\xi \ell(X^*_{x,T})^\top u_{x,T} \right] 
      + \red  \delta \nc \left(\E_{x \sim \mu_0} \left[ \lVert P^*_{x,0}\rVert^q  \right]\right)^{\red - \nc\frac 1 p} \E_{x \sim \mu_0} \left[ \lVert P^*_{x,0}\rVert^{q-2} (P^*_{x,0})^\top v_{x,0} \right] \notag \\
      %\frac{\sqrt{\delta} \langle (P_0^{*}(.))^\top v_0(.),\mu_0 \rangle }{2\sqrt{ \langle |P_0^{*}(.)|^2,\mu_0 \rangle} } \notag \\
      & = \E_{x \sim \mu_0} \left[ \nabla_\xi \ell(X^*_{x,T})^\top u_{x,T}+ (\beta^*_{x,0})^\top v_{x,0}\right]  
      \label{pontryagin1}
    \end{align}
    On the other hand, $\dd( (\beta^*_{x,t})^\top v_{x,t} ) = \dd (\beta^*_{x,t})^\top v_{x,t} +  (\beta^*_{x,t})^\top \dd v_{x,t} $, which implies after some cancellations that
    \begin{align*}  
      (\beta^*_{x,0})^\top v_{x,0} & =  (\beta^*_{x,T})^\top v_{x,T} \red - \nc (\beta_{x,\tau}^*)^\top \{ D_\xi f(X_{x,\tau}^*, \theta^*_\tau) - D_\xi f(X_{x,\tau}^*, \eta)   \}^\top P_{x,\tau}^*  \\
      & \quad \red + \nc   \int_0^T (\beta_{x,t}^*)^\top (D_\xi^2 f(X^*_{x,t},\theta^*_t) u_{x,t})^\top P_{x,t}^* \dd t.
    \end{align*}
    Therefore, equation \eqref{pontryagin1} becomes
    \begin{align}
      0 & \leq  \E_{x \sim \mu_0} \left[ \{  \nabla_\xi \ell(X^*_{x,T})^\top- (\beta^*_{x,T})^\top D_\xi^2 \ell(X_{x,T}^*)\} u_{x,T} \red + \nc  \int_0^T (\beta_{x,t}^*)^\top (D_\xi^2 f(X^*_{x,t},\theta^*_t) u_{x,t})^\top P_{x,t}^* \dd t \right]  \notag \\
          & \qquad \red + \nc  \E_{x \sim \mu_0} \left[ \beta^{*\top}_{ x \nc, \tau} \{ D_\xi f(X_{x,\tau}^*, \eta) - D_\xi f(X_{x,\tau}^*, \theta^*_\tau)   \}^\top P_{x,\tau}^*\right] \notag\\
       & = \E_{x \sim \mu_0} \left[  - \nc \alpha_{x,T}^{*\top} u_{x,T} \red + \nc \int_0^T \beta_{x,t}^{*\top} (D_\xi^2 f(X^*_{x,t},\theta^*_t) u_{x,t})^\top P_{x,t}^* \dd t \red + \nc \beta^{*\top}_{x,\tau} \{ D_\xi f(X_{x,\tau}^*, \eta) - D_\xi f(X_{x,\tau}^*, \theta^*_\tau)   \}^\top P_{x,\tau}^* \right]   \notag \\
       & = \E_{x \sim \mu_0} \left[\red - \nc \alpha_{x,\tau}^{*\top} \{ f(X^{*}_{x,\tau}, \eta ) -  f(X^{*}_{x,\tau}, \theta^*_\tau) \} \red + \nc  \beta^{*\top}_{x,\tau} \{ D_\xi f(X_{x,\tau}^*, \eta) - D_\xi f(X_{x,\tau}^*, \theta^*_\tau)   \}^\top P_{x,\tau}^* \right]. \notag
    \end{align} 
    This deduces the maximum principle \eqref{Modified maximum principle} in the case without running costs. 

    % Clearly, we have that $X_t^{\tau,\epsilon} =X_t^*$ for $t<\tau-\epsilon$ since the control is unchanged for these times. At $t=\tau$

  \end{proof}

\section{Second order regularization}
\label{sec:SecondOrderReg}

% The idea presented in section \ref{Sec:Regularized Robust Optimal} can be developed further to include higher order expansions of the function $J(\tilde \mu, \theta)$ under the assumption that the function $j(\cdot ,\theta)$ is regular enough. However, unlike in the first order case discussed in section \ref{Sec:Regularized Robust Optimal}, an explicit higher order expansion for $\max_{\tilde \mu} J(\tilde \mu, \theta)$ that does not involve any maximization problems is in general difficult to obtain unless one restricts to specific regimes for the size of the gradient function $\nabla_x j(\cdot ,\theta)$ with respect to the parameter $\delta$. In what follows we will present a formal analysis to motivate the following optimization problems:
% \begin{equation}
% \label{eqn:SecondOrderVersion1}
% \min_{\theta \in \Theta} \left\{  J(\mu_0, \theta) + \delta \left(\E_{x \sim \mu_0} \left[ \lVert \nabla_x j(x,\theta)\rVert^q \right] \right)^{1/q} +  \frac{\delta^2}{2} \left(\E_{x \sim \mu_0} \left[ .... \right] \right) \right\},
% \end{equation}
% \begin{equation}
% \label{eqn:SecondOrderVersion2}
% \min_{\theta \in \Theta} \left\{  J(\mu_0, \theta) +  \frac{\delta^2}{2} (\E \left[ (\lambda_{max}(x,\theta))^{\tilde q} \right])^{1/\tilde{q}} \right\}.
% \end{equation}

We motivate now problems \eqref{eqn:SecondOrderVersion1} and \eqref{eqn:SecondOrderVersion2}. Continuing our computations from section \ref{sec:TaylorGeometry} we can write for every $\tilde \mu$ within $W_p$-distance $\delta$ from $\tilde \mu$:
\[  J(\tilde \mu, \theta) = J(\mu_0, \theta) + \int_{\R^d} \nabla_x j(x,\theta) \cdot V_0(x) d\mu_0(x)  + \frac{1}{2}  \int_{\R^d} (D^2_x j(x,\theta) V_0(x)) \cdot V_0(x) d\mu_0(x) +O(\delta^3), \]
if for example we assume that the function $j( \cdot,\theta)$ has bounded third order derivatives uniformly over $\theta \in \Theta$. In that case we could expect that up to an error of order $O(\delta^3)$ (independent of $\theta \in \Theta$), the problem
\[ \max_{\tilde \mu: \: W_p(\mu_0, \tilde \mu) \leq \delta} J(\theta,\tilde \mu ) \]
(an optimization problem over a curved manifold) is equal to
\begin{equation}
\label{eqn:SecondOrder}
    J(\mu_0, \theta) +  \max_{ V_0 : \:  \lVert V_0\rVert_{L^p(\R^d:\R^d, \mu_0 )} \leq \delta }   \left\{ \int_{\R^d} \nabla_x j(x,\theta) \cdot V_0(x) d\mu_0(x) + \frac{1}{2}  \int_{\R^d} (D^2_x j(x,\theta) V_0(x)) \cdot V_0(x) d\mu_0(x)   \right\},
\end{equation}
which is again an optimization problem over a flat Banach space. However, in contrast with problem \eqref{eqn:FirstOrderFlat}, problem \eqref{eqn:SecondOrder} does not have an explicit solution. What is more, in general, the correct expansion of \eqref{eqn:SecondOrder} in $\delta$ (up to order two) depends on the size of $\nabla_x j( \cdot,\theta)$ relative to $\delta$ as we illustrate with the following analogous finite dimensional problem.

\begin{remark}
Consider the following maximization problem in $\R^m$:
\begin{equation}
\label{eqn:SecondOrderToy}
 \max_{v \in \R^m \text{ s.t. } \lVert  v\rVert  \leq \delta } \left\{ b \cdot v + (Av) \cdot v \right\},
\end{equation}
where $A$ is an arbitrary $m \times m$ symmetric matrix (not necessarily with a sign) and $b$ is an arbitrary vector in $\R^m$. Notice that:

\begin{itemize}
    \item If $\delta$ is small enough and $\frac{\delta}{\lVert b  \rVert}=o(1)$, then the linear term dominates the problem and we can write:
    \[ \eqref{eqn:SecondOrderToy} =  \delta \lVert  b \rVert + \delta^2 \left( A\frac{b}{\lVert b \rVert} \right) \cdot \left( \frac{b}{\lVert b \rVert}\right)    + o(\delta^2). \] 
    This value is obtained by plugging the maximizer of the problem $\max_{v \in \R^m \text{ s.t. } \lVert  v\rVert  \leq \delta } \left\{ b \cdot v \right\}$ in the objective of \eqref{eqn:SecondOrderToy}.
    
    \item If $\delta$ is small enough and $\frac{\lVert b  \rVert}{\delta}=o(1)$, then the quadratic term dominates the problem and we can actually write:
    \[ \eqref{eqn:SecondOrderToy} =  \delta^2 (\lambda_{max})_+  + o(\delta^2), \] 
    where in the above $\lambda_{max}$ is the largest eigenvalue of $A$ and $(a)_+$ denotes the positive part of $a\in \R$. This value is obtained by plugging the maximizer of the problem $\max_{v \in \R^m \text{ s.t. } \lVert  v\rVert  \leq \delta } \left\{ (Av) \cdot v \right\}$ in the objective of \eqref{eqn:SecondOrderToy}.
    \item When $\lVert b \rVert \sim \delta$, an explicit second order expansion for \eqref{eqn:SecondOrderToy} is intractable for all practical purposes as can be easily seen by inspection after writing the KKT conditions for this in general non-convex problem.
\end{itemize}

\end{remark}

To connect problems \eqref{eqn:SecondOrderVersion1} and \eqref{eqn:SecondOrderVersion2} with the previous remark we use the following observations. First, if we plug the $V_0$ from \eqref{eqn:V0Max} (the maximizer of the problem \eqref{eqn:FirstOrderFlat}) in the objective function from problem \eqref{eqn:SecondOrder} we obtain the objective function in problem \eqref{eqn:SecondOrderVersion1}. As for the objective in problem \eqref{eqn:SecondOrderVersion2} we notice the following.

\begin{proposition}
Let $p\geq 2$. Then, for every $\theta \in \Theta$ we have:
\begin{equation}
\label{eqn:SecondOrderSmallGradient}
 \max_{ V_0 : \:  \lVert V_0\rVert_{L^p(\R^d:\R^d, \mu_0)} \leq \delta  } \left\{   \frac{1}{2}  \int_{\R^d} (D^2_x j(x,\theta) V_0(x)) \cdot V_0(x) d\mu_0(x) \right\}   = \left(\int_{\R^d} |(\lambda_{max}(x,\theta))_+|^{\tilde q} d \mu_0(x) \right)^{1/\tilde q},
\end{equation}
where $\lambda_{max}(\theta, x)$ is the largest eigenvalue of $D^2_xj(x,\theta)$ and where $\tilde q$ is the conjugate of $p/2$, i.e.:
\[ \frac{2}{p}+ \frac{1}{\tilde q} = 1.\]
\end{proposition}

\begin{proof}
 
For each $x$ in the support of $\mu_0$ we select $V_0(x) = g(x)  U(x)$ where $U(x)$ is a unit (Euclidean) norm eigenvector of $D^2_x j(x,\theta)$ with eigenvalue $\lambda_{max}(\theta, x)$ and $g$ is a scalar function that satisfies $g(x)=0$ if $\lambda_{max}(x,\theta) \leq 0$ and $\int_{\R^d} |g(x)|^pd\mu_0(x) dx \leq \delta^p$ . Plugging this $V_0$ in the objective function of the max problem in \eqref{eqn:SecondOrderSmallGradient} we obtain:
\[  \frac{1}{2}  \int_{\R^d}  (\lambda_{max}(\theta, x))_+ (g(x))^2 d\mu_0(x).\]
It is then straightforward to show that:
\begin{align*}
\max_{  \lVert V_0\rVert_{L^p(\R^d:\R^d, \mu_0)} \leq \delta } & \left\{   \frac{1}{2}  \int_{\R^d} (D^2_x j(x,\theta) V_0(x)) \cdot V_0(x) d\mu_0(x) \right\} \\ &= \max_{ \lVert  g^2\rVert_{L^{p/2}(\mu_0)}  \leq \delta^2 } \left\{   \frac{1}{2}  \int_{\R^d}   (\lambda_{max}(\theta, x))_+ (g(x))^2 d\mu_0(x)\right\}.
\end{align*}
Given that the scalar function $x \mapsto (\lambda_{max}(x,\theta))_+$ is non-negative, we can recognize, by duality, that the right hand side of the above expression is equal to: 
\[ \frac{\delta^2}{2}\left(\int_{\R^d} |(\lambda_{max}(x,\theta))_+|^{\tilde q} d \mu_0(x) \right)^{1/\tilde q}, \]
obtaining in this way the desired result.
\end{proof}

In summary, problems \eqref{eqn:SecondOrderVersion1} and \eqref{eqn:SecondOrderVersion2} can be interpreted as second order expansions for $\inf_{\theta \in \Theta} \max_{\tilde \mu \: : W_p(\mu_0, \tilde \mu) \leq \delta} J(\tilde \mu, \theta)$ in two distinct regimes: 1) when gradients are \textit{not} small relative to $\delta$, more precisely, when the norms $\lVert  \nabla_x j(x,\theta)\rVert_{L^q(\R^d:\R^d, \mu_0)}$ are larger than $\delta$, and 2) when gradients are considerably smaller than $\delta$. In the next section we discuss the structure of both of these regularization problems. Recall that, as discussed in section \ref{sec:SecondOrderRegResults}, both of these problems are closely connected to problems used in the literature to train robust neural networks.

\nc

\subsection{ Adjoint variable formulation of the second-order regularized robust control problems}

% Following our reasoning for the first order robust control problem, we raise the question on whether we can find a Pontryagin like principle for the second order robust control problem. Let us remark, though, that for this principle to be relevant in practice, it should involve a number and dimensionality of variables that does not grow too much.

% \red
% Throughout this section, we limit the scope of our results to the case where there is no running cost.%, we use to define the robustness ball the Wasserstein distance with $p=2$, and we set $q,r = 2$. 
% This allows us to spare the reader of the  additional notational complexity required in the general case, which is otherwise addressed similarly.
% \nc

The main purpose in this section is to examine the transformations on the second order problems \eqref{eqn:SecondOrderVersion1} and \eqref{eqn:SecondOrderVersion2} in terms of adjoint variables.

\subsection{ The problem in Proposition \eqref{Prop:SecondOrder1WithAdjoints}} 

We aim to deduce Proposition \ref{Prop:SecondOrder1WithAdjoints}. We start by studying how to write the derivative of the regularization term in the first order expansion problem.

\begin{lemma}
Under Assumption \ref{A2}, and assuming that $p\geq 2$ and that $\lVert \cdot \rVert$ is the Euclidean norm we have:
\[  \nabla_x \lVert P_{x,0}\rVert^q =  - q \kappa_{\mu_0,\delta}  \hat \alpha_{x,0} \]
with $\kappa_{\mu_0,\delta} = \frac{1  }{\delta}  \left(\E_{x_0\sim \mu_0} [\lVert P_{x_0,0}\rVert^q] \right)^{1/p}$.
\label{Lem: alpha_hat_is_derivative}
\end{lemma} 

\begin{remark}
  The role of the constant $\kappa_{\mu,\delta}$ is to cancel the terms in the adjoint variables related to the mean-field contribution of the robust problem to allow us to focus on the pointwise result.
\end{remark}

\begin{remark}
  We had already seen for the original problem with loss function $j$ that $P_{x,0} = - \nabla_x (j(x,\theta))$, i.e., the direction of steepest descent of the loss function with respect to the initial point. Given the decomposition for the adjoint variable $\alpha$,  Lemma \ref{Lem: alpha_hat_is_derivative} implies that $\alpha$ plays the analogous role for the first order robust control problem. 
\end{remark}

\begin{proof}[Proof of Lemma \ref{Lem: alpha_hat_is_derivative}]
  For an arbitrary $\Delta x$ of unitary norm in $\R^d$, we have that
  \begin{align*}
    \nabla_x \lVert P_{x,0}\rVert^q \cdot \Delta x & = \lim_{\epsilon \downarrow 0} \frac{1}{\epsilon} [ \lVert P_{x+\epsilon \Delta x,0} \rVert^q - \lVert P_{x,0}\rVert^q   ]  = q \lVert P_{x,0}\rVert^{q-2} (P_{x,0})^\top \eta_0^{\Delta x},
  \end{align*}
  where 
  \begin{equation}
    \eta_t^{\Delta x} = \lim_{\epsilon \downarrow 0} \frac{1}{\epsilon} [ P_{x+\epsilon \Delta x,t} - P_{x,t} ].  
    \label{Def eta}  
  \end{equation}  
  
  From Assumption \ref{A2} and Leibniz rule, we get that
  \begin{align}
    \eta^{\Delta x}_t  & = - D^2_\xi \ell(X_{x,T},\theta_T) \zeta_T^{\Delta x} + \int_t^T D_\xi f(X_{x,s},\theta_s)^\top \eta^{\Delta x}_s + (D^2_{\xi} f(X_{x,s}, \theta_s) \zeta^{\Delta x}_s)^\top P_{x,s}   \dd s,
    \label{eta delta} 
  \end{align}
  where $\zeta^{\Delta x}$ is defined in \eqref{Directional_derivative_X}. Using the definition of the process $\beta$ in \eqref{Def beta}, it follows that
  \[ \nabla_x \lVert P_{x,0}\rVert^q \cdot \Delta x  =  q \lVert P_{x,0}\rVert^{q-2} (P_{x,0})^\top \eta_0^{\Delta x} = q\kappa_{\mu_0,\delta} (\beta_{x,0})^\top \eta_0^{\Delta x} .\]
  We can now use the integration rule for products and the dynamics for $\beta$ and $\eta^{\Delta x}$ to get
  
  \begin{align}
    \nabla_x \lVert P_{x,0}\rVert^q \cdot \Delta x & = q\kappa_{\mu_0,\delta} \left(  (\beta_{x,T})^\top \eta_T^{\Delta x} - \int_0^T (\eta_s^{\Delta x})^\top \dot{\beta}_{x,s} \dd s -  \int_0^T (\beta_{x,s})^\top \dot{\eta}^{\Delta x}_{s} \dd s \right) \notag\\
    & = q\kappa_{\mu_0,\delta} \left(  -(\beta_{x,T})^\top D^2_\xi \ell(X_{x,T},\theta_T) \zeta_T^{\Delta x}  - \int_0^T (\eta_s^{\Delta x})^\top \dot{\beta}_{x,s} \dd s -  \int_0^T (\beta_{x,s})^\top \dot{\eta}^{\Delta x}_{s} \dd s \right)\notag\\
    & = q\kappa_{\mu_0,\delta} \Big(  -(\hat \alpha_{x,T})^\top \zeta_T^{\Delta x}  - \int_0^T (\eta_s^{\Delta x})^\top  D_\xi f (X_{x,s},\theta_s) \beta_{x,s} \dd s \notag \\
    & \qquad  \qquad +  \int_0^T (\beta_{x,s})^\top D_\xi f(X_{x,s},\theta_s)^\top \eta^{\Delta x}_s \dd s \notag\\
    & \qquad  \qquad + \int_0^T (\beta_{x,s})^\top (D^2_{\xi} f(X_{x,s}, \theta_s) \zeta^{\Delta x}_s)^\top P_{x,s}  \dd s  \Big) \notag \\
    & = q\kappa_{\mu_0,\delta} \Big(  -(\hat \alpha_{x,T})^\top \zeta_T^{\Delta x}   + \int_0^T (\beta_{x,s})^\top (D^2_{\xi} f(X_{x,s}, \theta_s) \zeta^{\Delta x}_s)^\top P_{x,s}  \dd s  \Big). \label{eq_deriv_1}
  \end{align}
  Similarly, from the dynamics of $\hat \alpha$ and $\zeta^{\Delta x}$ we get
  \begin{align*}
    (\hat \alpha_{x,T})^\top \zeta_T^{\Delta x} & = (\hat \alpha_{x,0})^\top \zeta_0^{\Delta x} + \int_0^T  (\hat \alpha_{x,s})^\top D_\xi f(X_{x,s}, \theta_s) \zeta_s^{\Delta x} \dd s\\
    & \qquad - \int_0^T (\zeta_s^{\Delta x})^\top D_\xi f(X_{x,s},\theta_s)^\top \hat \alpha_{x,s} \dd s +  \int_0^T (\zeta_s^{\Delta x})^\top \{D^2_\xi f(X_{x,s},\theta_s)^{\red \top \nc} P_{x,s}  \} ^\top \beta_{x,s} \dd s  \\
    & = (\hat \alpha_{x,0})^\top \Delta x + \int_0^T (\zeta_s^{\Delta x})^\top \{D^2_\xi f(X_{x,s},\theta_s)^{\red \top \nc} P_{x,s}  \} ^\top \beta_{x,s} \dd s.
  \end{align*} 
   Replacing back into  \eqref{eq_deriv_1}, we conclude that
   \[ \nabla_x \lVert P_{x,0}\rVert^q \cdot \Delta x  =  -q\kappa_{\mu_0,\delta} \hat \alpha_{x,0} \cdot \Delta x  \]
   from where the claim follows.
\end{proof}

A straightforward consequence of Lemma \ref{Lem: alpha_hat_is_derivative} and the analogous result for $P$ is that the original cost function $j(\cdot ,\theta)$ is twice differentiable in the direction of the gradient. More precisely we obtain the following result.

\begin{corollary}
  Under Assumption \ref{A2}, $j(\cdot ,\theta)$ is twice differentiable in $x$ for any fixed control $\theta$ and
  \[ \lVert \nabla_x j(x,\theta)\rVert^{q-2} D^2_x j(x,\theta) \nabla_x j(x,\theta) =  \kappa_{\mu_0,\delta}  \hat \alpha_{x,0}.\]
  \label{Cor:alpha_hat_is derivative}
\end{corollary}

\begin{proof}[Proof of Proposition \ref{Prop:SecondOrder1WithAdjoints}]
  It follows directly from Corollary \ref{Cor:alpha_hat_is derivative} and the fact that $ P_{x,0}=-\nabla_xj(x,\theta)$ in \eqref{eqn:SecondOrderVersion1}.
\end{proof}

\subsection{ The problem in Proposition \eqref{Prop:SecondOrder2WithAdjoints}} 

As in the previous case, we start examining the role of the adjoint variables that we introduce into the problem.

\begin{lemma}
  For a fixed vector $v \in \R^d$, consider the adjoint variables
  \[\gamma _{x,t}^{v}=v+\int ^{t}_{0}D_{\xi }f\left( x_{x,s},\theta _{s}\right) \gamma _{x,s}^{v}ds,\]
    and    
    \[\begin{aligned}\rho^v_{x,t}= D_{\xi }^{2}\ell\left( X_{x,T}  ,\theta _{T }\right) \gamma _{x,T}^{v}+\int _{t}^{T}D_{\xi }f\left( X_{x,s},\theta_s\right)^{\red \top \nc} \rho^v_{x,s}ds\\
    -\int ^{T }_{t}\left\{ D_{\xi }^{2}f\left( X_{x,s},\theta_s\right) ^{\top}P_{x,s}\right\} ^{\top}\gamma_{x,s}ds.\end{aligned}\]
  Then
  \[ \nabla_x(P_{x,0} \cdot v) = - \rho_{x,0}^v.  \]
  \label{lem:adjoint_second_direction_v}
\end{lemma}

\begin{proof}
The proof of this result is very similar to that of Lemma \ref{Lem: alpha_hat_is_derivative} and thus we skip the details.
\end{proof}

\nc

\begin{proof}[Proof of Proposition \ref{Prop:SecondOrder2WithAdjoints}]
  On the one hand, re-expressing the problem in terms of the adjoint variables is a direct consequence of Lemma \ref{lem:adjoint_second_direction_v}.
  
  On the other hand, under the stated assumptions, the argument inside the expectation in \eqref{eqn:SecondOrderVersion2Modified} has finite variance uniformly in $\theta$ and $x$: therefore, replacing the expectation by the empirical mean using $m$ samples produces an estimator that converges almost surely by the law of large numbers and has errors subject to the central limit theorem. 
\end{proof}

  \section{Training robust neural networks}
  \label{sec:algorithm}
  
%   \blue 
%   Training a neural network to solve the robust problem \eqref{Robust problem:Intro} can be very time-consuming: the min-max problem needs to be typically addressed using methods like \red adversarial training \nc that, while useful, take a lot of time and resources (energy and memory) to be applied. 
  
%   \nc
  
  An approach to robust training is suggested by the results we have presented on the regularized adversarial control problems. The Pontryagin principle in Theorem  \ref{Thm: Pontryagin} can be used to create optimization algorithms: training the network can be understood as solving a fixed-point problem where the constraints in \eqref{eqn:ControlExpanded},  equations \eqref{Def alpha}, \eqref{Def beta}, and the maximum principle \eqref{Modified maximum principle} must be simultaneously satisfied. There are many methods to solve numerically such a fixed point problem, but, undoubtedly, the most popular consists in applying consecutively a step of forward propagation to solve for the primal variables, a step of backward propagation to get the dual variables, and the solution of an optimization  algorithm to update the controls (typically this is substituted with a gradient step to solve such optimization problem with an approach like stochastic gradient descent or any of its siblings). 
  
  We present in Algorithm \ref{Algo 1} an implementation of the first order regularized control problem applied to ResNets for the case with no running cost. The adjustment to the case with running cost is straightforward. Note that all equations, except for the one of $X$, are linear in their respective variables. Moreover, they involve only $f$ and its two derivatives. Thus, we can easily implement this algorithm in platforms like TensorFlow or PyTorch.

  \begin{algorithm}
    \caption{Backpropagation with SGD for robust control problem - ResNet }
    \begin{algorithmic}[1]
      \State Set $h, \gamma $ small constants
      \State $i=0$
      \State Initialize $\theta^0_k \equiv 0 $ for all $k$
      \While{No convergence}
      \For{Every batch}   
      \State Set $X_0=x_0$ for each $x_0$ in the batch
      \State Forward propagate using activation function ($X$):
      \State \hspace*{2em} $X_{k+1} = X_k + h f(X_k,\theta_k) $ 
      \State Backpropagate using derivatives of activation functions ($P$):
      \State \hspace*{2em} Set $P_{N} = -D_\xi \ell(X_N) $
      \State \hspace*{2em} $P_{k} = (I \red + \nc  h D_\xi f(X_k,\theta_k)^\top) P_{k+1} $ 
      \State Forward propagate using derivatives of activation functions ($\beta$):
      \State \hspace*{2em} Set $\beta_0 = \delta  \left( \tilde{\E} \left[ \lVert P_0\rVert^q  \right] \right)^{- \frac 1 p}  \lVert P_0\rVert^{q-2} P_0 $; where $\tilde{\E}$ is mean over elements in the batch.
      \State \hspace*{2em} $\beta_{k+1} = (I+ h D_\xi f(X_k,\theta_k) ) \beta_{k} $ 
      \State Backpropagate using first and second derivatives of the activation function ($\alpha$):     
      \State \hspace*{2em} $\alpha_T = -D_\xi \ell(X_N) + D_\xi^2\ell(X_N) \beta_N$    
      \State \hspace*{2em} $\alpha_{k} = (I \red + \nc h D_\xi f(X_k,\theta_k)^\top )\alpha_{k+1} \red - \nc h P_k^\top D_\xi^2f(X_k,\theta_k) \beta_k $
      \State Calculate for each $k$ the gradient:
      \State \hspace*{2em} $\nabla_\vartheta H(X_k,P_k,\alpha_k,\beta_k,\vartheta^i_k) = \alpha_k \cdot D_\vartheta f(X_k,\vartheta_k^i) \red - \nc \beta_k \cdot ( D_{\vartheta,\xi} f(X_k,\vartheta_k^i)^\top P_k   )$      
      %P_k \cdot D_\vartheta f(X_k,\vartheta^i_k) - \nabla_\vartheta \Phi(X_k,\vartheta^i_k).$      
      \State Update the control for each $k$:
      \State \hspace*{2em} $ \theta_k = \theta_k \red + \nc \gamma \tilde{\E} [\nabla_\vartheta H(X_k,P_k,\alpha_k,\beta_k,\theta_k) ]$ \label{line_update_theta}
      \State $i=i+1$
      \EndFor            
      \EndWhile
    \end{algorithmic} 
    \label{Algo 1}
  \end{algorithm}

  % Keeping in mind the connection between the adjoint variable and the derivative of the objective function with respect to the primal variable at time $t$, we can re-interpret the expressions the derivative of the outcome with respect to the control at a given stage. This is precisely what is commonly known as backpropagation training.   

  % Motivated by these ideas and Theorem \ref{Thm: Pontryagin}, we can propose an algorithm for robust training. Given that forward neural networks are more popular than ResNets, we introduce the algorithm directly for the former case. For simplicity, we consider the case where $\Phi\equiv 0$. The algorithm is based on keeping track of the (robust) primal and dual variables $X,P,\alpha,\beta$ following \eqref{eqn:ControlExpanded},\eqref{Def alpha} and \eqref{Def beta}. The scheme of the algorithm is presented in \ref{Algo 1}.

  The algorithm takes an even simpler form when considering a ResNet with ReLu activation functions at each stage: although this activation function is not differentiable, the backpropagation  algorithm has been successfully applied using a 'relaxed' gradient. Following the same ideas, Algorithm \ref{Algo 1} follows with $D^2_\xi f(\xi,\vartheta) = 0$. In this particular case, all linear equations are driven by the same factor $\nabla_\xi f$, which makes it simpler to implement.  

  \begin{remark}
  \label{rem:DoubleBackProp}
    Let us stress (as has been done before, for example in \cite{e_meanfield_2018,EMaximPrinciple}) that backpropagation training is by no means the only possible approach to solve the fixed-point problem for training, and in certain problems can have structure favouring alternative algorithms.   
    We notice that Algorithm \ref{Algo 1} in the case $p=\infty$ (and $q=1$) is the \textit{double backpropagation} algorithm from \cite{DoubleBackProp}. This follows from the discussion presented in section 4.2 in \cite{EMaximPrinciple} on the general relation between the method of successive approximations and gradient descent with backpropagation. 
  \end{remark}
    
  \begin{remark}
    In strict terms, we have results that are applicable only to ResNets. However, they can be formally generalized to other types of neural networks. For instance, one can rewrite an instance of a vanilla forward network in terms of a ReNet by setting the diffusion coefficient to be
      \[ \tilde f(t_k, X_k) = \frac{f(t_k,X_k) -X_k }{h},\]
    where $h$ is the small coefficient in \ref{Algo 1} representing the time discretization. The net effect on Algorithm \ref{Algo 1} is that equations are no longer residual.
  \end{remark}

  We present in Algorithm \ref{Algo 2} the implementation of second order Pontryagin principle in \eqref{eqn:ControlExpandedSecondOrderWithAdjoints}, still in the case of ResNet. Note that there are strong similarities between the forward and backpropagation of two pairs of variables, which is of advantage for any possible implementation. Anologously to the case of Algorithm \ref{Algo 1}, additional simplifications can be obtained for the case of activation functions like ReLu.

    \begin{algorithm}
      \caption{Backpropagation with SGD for robust control problem - second order }
      \begin{algorithmic}[1]
        \State Set $\gamma $ small constant (learning rate)
        \State $i=0$
        \State Initialize $\theta^0_k \equiv 0 $ for all $k$
        \While{No convergence}
        \For{Every batch}   
        \State Set $X_0=x_0$ for each $x_0$ in the batch
        \State Forward propagate using activation function ($X$):
        \State \hspace*{2em} $X_{k+1} = X_k + h f(X_k,\theta_k) $ 
        \State Backpropagate using derivatives of activation functions ($P$):
        \State \hspace*{2em} Set $P_{N} = -D_\xi \ell(X_N) $
        \State \hspace*{2em} $P_{k} = (I \red + \nc  h D_\xi f(X_k,\theta_k)^\top) P_{k+1} $ 
        \State Forward propagate using derivatives of activation functions ($\beta, \lambda$): 
        \State \hspace*{2em} $\beta_0 =  \delta  \left( \tilde{\E} \left[ \lVert P_0\rVert^q  \right] \right)^{- \frac 1 p}  \lVert P_0\rVert^{q-2} P_0 $; where $\tilde{\E}$ is mean of elements in the batch. \label{line_beta0}
        \State \hspace*{2em} $\beta_{k+1} = (I+ h D_\xi f(X_k,\theta_k) ) \beta_{k} $ 
        \State \hspace*{2em} Set $\lambda_0 = \frac \delta 2 \frac{P_0}{\lVert P_0\rVert}$
         \label{line_lambda0}
        \State \hspace*{2em} $\lambda_{k+1}  = (I+ h D_\xi f(X_k,\theta_k)^\top ) \lambda_{k} $ 
        \State Backpropagate using first and second derivatives of the activation function ($\alpha, \psi$):    
        \State \hspace*{2em} Set $\alpha_T = D_\xi^2\ell(X_N) \beta_N$    
        \State \hspace*{2em} $\alpha_{k} = (I \red + \nc  h D_\xi f(X_k,\theta_k)^\top ) \alpha_{k+1} \red - \nc  P_k^\top D_\xi^2f(X_k,\theta_k) \beta_k $
        \State \hspace*{2em} Set $\psi_T = - D_\xi^2\ell(X_N)^\top \beta_N$    
        \State \hspace*{2em} $\psi_{k} = (I+h D_\xi f(X_k,\theta_k)) \psi_{k+1} \red + \nc   P_k^\top D_\xi^2f(X_k,\theta_k) \lambda_k   $
        \State Forward propagate using first and second order derivatives ($\pi$):
        \State \hspace*{2em} Set $\pi_{0}  = \frac 1 {\lVert P_{0}\rVert} \left( \red P_0  \nc + \frac \delta 2 \alpha_{0} - \psi_{0}  \right) - \frac {\red \delta \nc} {\lVert P_{0}\rVert^3} P_{0} \cdot \left( \frac {\red 1 \nc } 2 \alpha_{0} - \psi_{0}  \right)  P_{0}$
        \State \hspace*{2em} $\pi_{k+1}  = (I \red + \nc D_\xi f(X_{k},\theta_k)) \pi_k  - \{ (\beta_{k})^\top (D^{\red 2 \nc}_\xi f(X_{k},\theta_k))^\top  \}^\top \lambda_k$
        \State Backpropagate using first, second, and third order derivatives ($\psi$):
        \State \hspace*{2em} $\phi_{N}  =  P_{N} \red + D^2_x \ell(X_N) \pi_N \nc-\{D_\xi^3 \ell(X_{N}, \theta_{N}) \beta_{N} \}^\top \lambda_{N}$
        \[ \begin{split} \phi_{k}  & = (I \red + \nc  D_\xi f(X_{k},\theta_k)^\top) \phi_{k+1} \red -\nc  \{D^2_\xi f(X_{k},\theta_k)^\top P_{k}   \}^\top  \pi_{k}  \\ & - \{D^2_\xi f(X_{k},\theta_k)^\top \alpha_{k}   \}^\top  \lambda_{k} \red + \nc  \{D^2_\xi f(X_{k},\theta_k)^\top \beta_{k}   \}^\top  \psi_{k} \\ & \red + \nc (\{D^3_\xi f(X_{k},\theta_k)^\top P_{k}   \}^\top \beta_{k})^\top  \lambda_{k} \end{split}\]
        \State Calculate for each $k$ the gradient:         
        \[ \begin{split} 
          \nabla_\vartheta H(X_k,P_k,\alpha_k,\beta_k,\theta_k) & = \phi_k^\top  D_\vartheta f(X_k,\theta) \red - \nc P_k^\top  D_\xi f(X_k,\theta_k) \pi_k \red - \nc \alpha_k^\top  (D_{\vartheta,\xi} f(X_k,\theta_k)) \lambda_k   \\
          & + \lambda_k^\top  \{   (D_{\vartheta,\xi,\xi} f(X_k,\theta_k)^\top P_k\}^\top  \beta_k +  \beta_k^\top (D_{\vartheta,\xi} f(X_k,\theta_k)) \psi_k
        \end{split}\]     
        %P_k \cdot D_\vartheta f(X_k,\vartheta^i_k) - \nabla_\vartheta \Phi(X_k,\vartheta^i_k).$      
        \State Update the control for each $k$:
        \State \hspace*{2em} $ \theta_k = \theta_k \red + \nc \gamma \tilde{\E} [\nabla_\vartheta H(X_k,P_k,\alpha_k,\beta_k,\theta_k) ]$ \label{line_update_theta2}
        \State $i=i+1$
        \EndFor              
        \EndWhile
      \end{algorithmic} 
      \label{Algo 2}
    \end{algorithm}     

    Algorithm \ref{Algo 1} and \ref{Algo 2} are written in accordance to the most typical way of implementing training algorithms via the use of batched optimization.  In this type of implementation, the training sample is subdivided in different batches: we calculate all propagation equations for initial points in the subsample values and then update the control $\theta_k$ using the empirical expectation calculated with the points in the sample. 

    Frequently, practitioners use rather small values for the batch size (referred in those cases as \emph{mini-batches}). Note, however, that the mean-field effect of the optimization implies that in addition to the control updating term, we also need to calculate an average in the term $\tilde{\E} \left[ \lVert P_0\rVert^q  \right]$ in line \ref{line_beta0} to initialize $\beta_0$. For stability reasons we therefore advice against using small batch sizes in this case. A notable exception appears when $p =  \infty$ (i.e. $q = 1$), when the initialization of $\beta$ in line \ref{line_beta0} becomes $\beta_0 = \sign(P_0)$. Thus, in this case, the mean-field action disappears and small batches are again perfectly acceptable. 

  %  backpropagation algorithms implement, in discrete time: suppose $-P$ is calculated using an Euler discretization 
  % \begin{equation} 
  %   -P_{k} = (I+ h D_\xi f(X_k,\theta_k) ) (-P_{k+1}) +  h \nabla_\xi \Phi(X_k,\theta_k) 
  %   \label{P discrete}
  % \end{equation}
  % with final condition $-P_T = D_\xi \ell(X_N)$, where we assumed we used a time grid with $N$ steps of size $h$.

  \red
  \section{Numerical illustration}
\label{sec:Numerics}

   We illustrate our results numerically in the context of image classification. We train a simple convolutional network\footnote{Two layers with a convolutional kernel, ReLu activation functions, and maxpool; and two linear layers at the end} to perform the classification task on the MNIST database. We then test the network with a \emph{clean} testing sample, and with an \emph{adversarial} version constructed via modification of the latter using PGD with 20 steps and a step size of 0.04.
   
   We train the network in three different versions: the \emph{baseline} method (i.e. unrobust) which uses the cross-entropy loss function,  and the \emph{Order 1} and \emph{Order 2} versions obtained by adding regularization terms as explained in problems \ref{eqn:ControlExpanded} and \ref{eqn:ControlExpandedSecondOrderWithAdjoints}, respectively. As parameters, we fix $\delta=0.2$, $p=\infty$ (equivalently, $q=1$), and we take the norm $\|\cdot \|_*$ to be the $r-$ norm with $r\in \{2, \infty\}$.
    
    Table \ref{Tbl:Accuracy} shows the accuracy of the network after training with the three stated procedures. Although all training procedures perform similarly when evaluated with a clean testing sample, the Order 1 regularization significantly improves the robustness of the network when subject to a sample modified by the adversarial attack. The table also shows that the choice of the vector norm to be used plays a significant role. This is not a surprise, since one can understand the PGD attack as directed by taking infinity norms on successive gradients. Different choices of norms might be better suited for other types of adversarial attacks. Importantly, the improved robustness comes with a moderate cost in training time of 14\% (39\% for Order 2) over the baseline training time. 
    
      \begin{table}
    \begin{tabular}{|c|ccc|ccc|}
      \hline
      & \multicolumn{3}{ c |}{$r=2$} & \multicolumn{3}{ |c| }{$r=$inf} \\
      %\cline{2-7}
      & Baseline & Order 1 & Order 2 & Baseline & Order 1 & Order 2\\
      \hline
      Accuracy (clean) & 98.41 & 98.42 & 98.42 & 98.41 & \textbf{98.45} &  98.44 \\
      Accuracy (adversarial) & 0.68&  2.09&  2.11& 0.7 & \textbf{23.11} & 22.94\\
      \hline
      Training time (factor)  & 1.0 & 1.14 & 1.39 & \textbf{0.99} & 1.13 & 1.38\\
      \hline
    \end{tabular} \caption{\red Accuracy and relative training time after 5 epochs. Although all training procedures are similarly capable when evaluated with a clean test, the regularization improves the resilience of the network when subject to adversarial attacks at a moderate cost. Note the dependence of results on the chosen vector norm $\lVert \cdot \rVert$. Best values in bold. \label{Tbl:Accuracy}}
  \end{table}
    
    Table \ref{Tbl:Accuracy} also shows that the Order 2 method does not seem to be contributing to the overall robustness improvement beyond what is already done by Order 1. In order to have a better understanding of the numerical effect of each procedure, we plot in Figure \ref{Fig:Surface} a local view of the loss surface obtained by calculating the cross-entropy for a clean testing image and perturbations around it. The corrupted images are obtained by an additive perturbation of the size marked in each axis: one in an adversarial direction and another in a random direction orthogonal to the adversarial. This way of illustrating the results is suggested in \cite{NEURIPS2019_7503cfac}. 
    
    The robustness effect of the regularized problems is illustrated in Figure \ref{Fig:Surface} by a reduction in values of the level of the surface, which translates in smaller cross-entropy and higher likeliness of obtaining an accurate classification. The plot also suggests that the effect of the Order 2 method is small compared to the Order 1 procedure, and mainly tends to reduce the curvature of the surfaces close to the \emph{clean} image. Hence, one would expect better robustness of Order 2 when considering directions not aligned with the adversarial one. In this sense, a more thorough study of the robustness induced by the Order 2 procedure away from the PGD line of attack would be interesting but outside of the scope of this work.
    
   \begin{figure}[h]
     \begin{center}
       \includegraphics[width=0.9\linewidth]{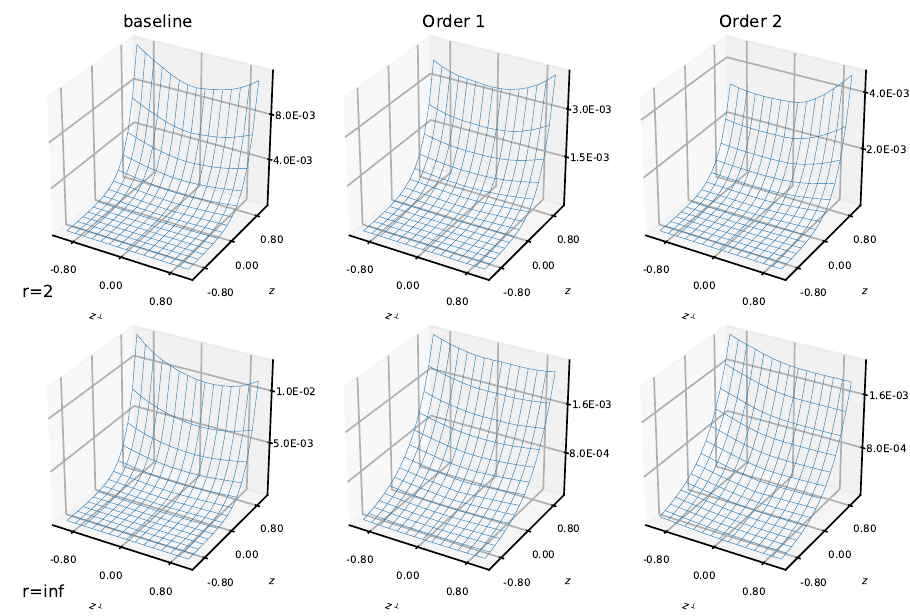}
     \end{center}
     \caption{\red Cross-entropy loss for one test image perturbed in an adversarial direction $z$ and a random one $z^\perp$ that is perpendicular to the adversarial direction. Top $r=2$, bottom $r=$inf. The first order regularization effect is manifested in a reduction in the loss function for perturbed images. The second order regularization effect, if still noticeable, is smaller in comparison, and tends to reduce slightly the curvature. \label{Fig:Surface}}
   \end{figure}
   
     \nc 
 
% 'train_time': array([[37.60194755, 42.79205108, 52.28909683],
%   [37.15215516, 42.38431001, 52.06984591]])

\section{Conclusions}
\label{sec:Conclusions}

In this paper we have established a series of connections between distributionally robust learning as modeled by a min-max problem of the form \eqref{Robust problem} and regularized risk minimization problems on the parameters of a deep ResNet neural network. To establish this connection, we study the $\max$ part of the min-max problem using tools from optimal transport theory and identify its leading order terms as a function of $\delta$, i.e. the power of the adversary. We remark that this approach is not restricted to adversarial problems on deep neural networks, and in particular can be used in other learning settings as long as the dependence of the loss function on the input data is regular enough. The specific ResNet deep neural network structure, however, allows us to interpret the resulting regularization problems as mean field optimal control problems. In turn, these control problems suggest, through their associated Pontryagin maximum principles, a family of algorithms for the training of robust neural networks which can avoid the computation of data perturbations during training. A key property of the resulting control problems is that they are scalable, and in particular the number and dimension of state variables is within a dimension-free factor of the dimensions of the original (unrobust) learning problem.

Some interesting research directions that stem from this work include: 1) Studying the type of regularity enforced on the input-to-output mappings by the regularization problems discussed in this paper. 2) The analysis of other distributionally robust problems where for example their objective target may contain an entropic term (as motivated in \cite{DongDP0020}). 3) In general, the use of tools from optimal control theory for the robust training of a wider class of neural networks. 4) The study of adversarial problems in other learning settings of interest where specific structure in the models may be exploited to get novel theoretical or algorithmic insights.

\medskip

\textbf{Acknowledgements:} The authors would like to thank two anonymous reviewers for their positive and constructive feedback. The authors would like to thank Leon Bungert for enlightening conversations and for providing them with many useful references. NGT was supported by NSF-DMS grant 2005797 and would also like to thank the IFDS at UW-Madison and NSF through TRIPODS grant 2023239 for their support. Part of this work was completed while NGT was visiting the Simons Institute to participate in the program ``Geometric Methods in Optimization and Sampling" during the Fall of 2021. NGT would like to thank the institute for hospitality and support.  

% \textbf{Conflict of interest statement:}  On behalf of all authors, the corresponding author states that there is no conflict of interest.

% Several theoretical questions arise from this work. One specific question is the following. The regularization term that appears in our analysis penalizes the parameters of a network, but it is not clear what is its explicit effect on the actual functions that penalized parameters induce. In particular, it is of interest to understand if penalization of parameters with the new regularization term does effectively translate into more regular mappings from inputs to outputs. We notice that here we have focused on \textit{deep} neural networks, but recent functional analytical results on the structure of shallow networks (in particular $1$ hidden layer networks) perhaps makes the regularity question more tractable in the shallow network setting.  In our sibling paper \cite{} we explore the regularity question in the control setting by studying the effect of regularization on generalized controls. For the different learning model of non-parametric binary classification, the regularity on input-output mappings has been studied recently in \cite{}, in particular deducing that explicit geometric quantities (perimeter and curvature of decision boundaries) are penalized in the presence of an adversary. This work and its sibling \cite{} are a first step in the attempt of understanding, from a mathematical perspective, the regularization effect of adversarial learning in the setting of neural networks.

\bibliography{ML}

\bibliographystyle{abbrv}

\end{document}